\documentclass[10pt]{article}
\usepackage{amsmath, amssymb, amsfonts, mathtools, amsthm}
\usepackage{hyperref}
\usepackage{cleveref}
\usepackage{subcaption}
\usepackage[numbers]{natbib}
\usepackage{algorithm}
\usepackage{algorithmic}
\usepackage{authblk}

\usepackage[margin=1in]{geometry}
\usepackage{graphicx}
\usepackage{booktabs}
\usepackage{xcolor}

\usepackage{algorithm}

\usepackage{multirow}
\usepackage{pifont}
\newcommand{\cmark}{\ding{51}}

\usepackage{aliascnt}

\newtheorem{theorem}{Theorem}[section]

\newaliascnt{lemma}{theorem}
\newtheorem{lemma}[lemma]{Lemma}
\aliascntresetthe{lemma}

\newaliascnt{definition}{theorem}
\newtheorem{definition}[definition]{Definition}
\aliascntresetthe{definition}

\newaliascnt{assumption}{theorem}
\newtheorem{assumption}[assumption]{Assumption}
\aliascntresetthe{assumption}

\newaliascnt{proposition}{theorem}
\newtheorem{proposition}[proposition]{Proposition}
\aliascntresetthe{proposition}

\newaliascnt{example}{theorem}

\aliascntresetthe{example}

\crefname{theorem}{theorem}{theorems}
\Crefname{theorem}{Theorem}{Theorems}
\crefname{lemma}{lemma}{lemmas}
\Crefname{lemma}{Lemma}{Lemmas}
\crefname{definition}{definition}{definitions}
\Crefname{definition}{Definition}{Definitions}
\crefname{assumption}{assumption}{assumptions}
\Crefname{assumption}{Assumption}{Assumptions}
\crefname{proposition}{proposition}{propositions}
\Crefname{proposition}{Proposition}{Propositions}
\crefname{example}{example}{examples}
\Crefname{example}{Example}{Examples}

\DeclareMathOperator{\diag}{diag}

\makeatletter
\newcommand{\alglinelabel}{
  \addtocounter{ALC@line}{-1}
  \refstepcounter{ALC@line}
  \label
}
\makeatother

\begin{document}
\title{Convergence Analysis of Randomized Subspace Normalized SGD under Heavy-Tailed Noise}
\author[1,2]{Gaku Omiya}
\author[2]{Pierre-Louis Poirion}
\author[1,2]{Akiko Takeda}

\affil[1]{Department of Mathematical Informatics, The University of Tokyo, Tokyo, Japan}
\affil[2]{Center for Advanced Intelligence Project, RIKEN, Tokyo, Japan}

\date{}
\date{}
\maketitle

\begin{abstract}
Randomized subspace methods reduce per-iteration cost; however, in nonconvex optimization, most analyses are expectation-based, and high-probability bounds remain scarce even under sub-Gaussian noise. We first prove that randomized subspace SGD (RS-SGD) admits a high-probability convergence bound under sub-Gaussian noise, achieving the same order of oracle complexity as prior in-expectation results. Motivated by the prevalence of heavy-tailed gradients in modern machine learning, we then propose randomized subspace normalized SGD (RS-NSGD), which integrates direction normalization into subspace updates. Assuming the noise has bounded $p$-th moments, we establish both in-expectation and high-probability convergence guarantees, and show that RS-NSGD can achieve better oracle complexity than full-dimensional normalized SGD.

\end{abstract}


\section{Introduction}
We consider the following nonconvex stochastic optimization problem:
\begin{align}
    \min_{x \in \mathbb{R}^d} F(x), \qquad
    F(x) := \mathbb{E}_{\xi \sim \mathcal{D}}[f(x,\xi)],
    \label{eq:sto_opt}
\end{align}
where $\mathcal{D}$ denotes an underlying data distribution. Nonconvex
optimization problems encountered in machine learning applications, such as training
deep neural networks, are often high-dimensional in recent years. 
Therefore, there is a growing need for nonconvex optimization algorithms that operate in a low-dimensional subspace at each iteration.

One possible way to reduce the dimension of \eqref{eq:sto_opt} is to work in random low-dimensional subspaces: Using a random matrix $P_k \in \mathbb{R}^{d \times r}$ with some chosen reduced dimension $r$ ($\ll d$), one approximates the full gradient $\nabla F(x_k)$ via directional derivatives along the columns of $P_k$ and performs the update with stepsize $\bar{\eta}$,
\[
  x_{k+1} = x_k - \bar{\eta} P_k P_k^\top \nabla F(x_k).
\]
This subspace-descent approach was developed as a deterministic optimization method minimizing $F$ by \citet{kozak2021}. This randomized subspace technique has also been applied to stochastic optimization using stochastic gradients $g_k$ instead of $\nabla F(x_k)$, where \citet{flynn2024} analyzes the convergence of stochastic gradient descent (SGD) with 
randomized subspace updates under the bounded-variance assumption. To handle deterministic large-scale problems, the optimization community has long relied on randomized coordinate descent methods \cite{nesterov2012,richtarik2014,wright2015}.
These methods update the solution by randomly selecting a single coordinate (or, more generally, a small set of $r$ coordinates) at each iteration.
The random-subspace framework above provides a broader view: it includes randomized (block) coordinate descent as a special case \cite{kozak2021}.
Moreover, \citet{kozak2021} derive high-probability guarantees only for Haar-random subspaces, which demonstrates the advantages of using Haar over random coordinate-descent-type schemes.

In machine learning, such randomized subspace methods can be used in at least two ways.
First, they can be used to reduce the memory footprint of gradient computations via forward-mode automatic differentiation (AD), which is particularly useful in large-scale models where activation memory dominates. 
 Computing the projected gradient $P_k^\top \nabla F(x_k)$ (or $P_k^\top g_k$)  via forward-mode AD requires only an $r/d$ fraction of the cost of computing the full-dimensional gradient $\nabla F(x_k)$ (or $g_k$).
Second, they can serve as a gradient compression technique in distributed optimization, alleviating communication bottlenecks by projecting gradients onto low-dimensional random subspaces before transmission.
In a related but distinct line of work, low-dimensional update ideas have also been explored for memory- and communication-efficient large-scale training, but their primary focus and algorithmic setting differ from ours; see \Cref{app:llm-subspace}.
We refer the reader to \Cref{app:related-work} for further details and background on this setting.

\begin{table*}[t]
\centering
\caption{
Convergence guarantees for stochastic gradient methods with randomized subspace updates (RS) and direction normalization (Norm).
Exp./w.h.p.: in expectation/high probability.
Noise: sub-Gaussian, BV, and $p$-BCM; see \Cref{ass:BV,ass:subgauss,ass:pBCM}, respectively.
}

\label{tab:comparison}
\begin{tabular}{llcc c cc c}
\toprule
Method &
& RS
& Norm
& Noise model
& Exp.
& w.h.p.
\\
\midrule
SGD & \cite{ghadimi2013}
& 
& 
& BV
& \cmark
& 
\\

SGD & \cite{liu2023hp}
& 
& 
& sub-Gaussian
& 
& \cmark
\\

RS-SGD & \cite{flynn2024}
& \cmark
& 
& BV
& \cmark
& 
\\

\textbf{RS-SGD} &  \textbf{(this work)}
& \cmark
& 
& sub-Gaussian
& 
& \cmark
\\
\midrule
NSGD & \cite{huebler2025}
& 
& \cmark
& p-BCM
& \cmark
& \cmark
\\

\textbf{RS-NSGD} & \textbf{(this work)}
& \cmark
& \cmark
& p-BCM
& \cmark
& \cmark
\\
\bottomrule
\end{tabular}
\end{table*}

\paragraph{New convergence analysis to randomized subspace SGD (RS-SGD): }
We investigate the theoretical properties of SGD combined with randomized subspace updates.
Recent work has studied stochastic methods that leverage random subspaces (e.g.,  \citet{flynn2024,han2025efficient}).
However, existing convergence analyses are mostly in expectation, and high-probability guarantees remain limited, largely due to technical challenges.
In this work, we establish high-probability convergence guarantees for stochastic gradient methods with randomized subspace updates under sub-Gaussian gradient noise.
While in-expectation guarantees capture average performance, they do not quantify the probability that a single run deviates substantially from this average; high-probability analyses address this by bounding such deviations with probability at least $1-\delta$ \citep{harvey2019,madden2024}.
In modern machine learning, repeating training across multiple independent runs can be costly, which increases the value of single-run guarantees \citep{liu2023hp}.
Compared to the analysis of \citet{flynn2024}, our results provide high-probability guarantees by adopting a sub-Gaussian noise assumption.
Our contribution is closely related in spirit to \citet{liu2023hp}, which analyzed standard SGD under sub-Gaussian noise; we extend this high-probability perspective to randomized subspace updates.
Now RS-SGD admits convergence guarantees in expectation under the bounded-variance assumption
and high-probability guarantees under a sub-Gaussian noise model.

However, these light-tailed noise conditions can be overly restrictive in modern machine-learning workloads, where stochastic gradients are often heavy-tailed \cite{simsekli2019}; this limitation is not caused by the randomized subspace projection but rather reflects the well-known sensitivity of SGD-type updates to heavy-tailed gradient noise \cite{nguyen2019,hodgkinson2021}.
To overcome this issue, we leverage \emph{normalized} stochastic gradient methods (NSGD),
which admit both in-expectation and high-probability convergence guarantees under the weaker bounded $p$-th moment
assumption stated later in \Cref{ass:pBCM} \cite{huebler2025}.
Unlike clipping-based methods, whose analyses typically require problem- or noise-dependent parameters,
including the tail index $p$ \cite{zhang2020clipped,nguyen2023}, NSGD can be analyzed without knowing $p$.
\paragraph{New randomized subspace method: Randomized Subspace Normalized SGD (RS-NSGD).}
Building upon the randomized subspace framework, we propose a subspace-compatible normalized method.
A direct plug-in hybrid such as
\[
x_{k+1} = x_k - \bar{\eta}\,\frac{P_kP_k^\top g_k}{\|g_k\|}
\]
still requires the full-dimensional norm $\|g_k\|$, and thus falls outside the randomized-subspace model we target.
In contrast, RS-NSGD normalizes using only the available projected gradient $P_k^\top g_k$ and performs the iteration
\[
x_{k+1} = x_k - \bar{\eta}\,\frac{P_kP_k^\top g_k}{\|P_k^\top g_k\|},
\]
thereby preserving the low-dimensional update structure of randomized subspace algorithms while gaining robustness to
heavy-tailed noise.
Because normalization introduces a nonlinear dependence on the random quantity $\|P_k^\top g_k\|$,
existing analytical tools for randomized subspace SGD do not directly apply, substantially complicating the analysis.

Our main contributions are as follows and are summarized in \Cref{tab:comparison}:
\begin{itemize}
   \item \textbf{Benefit of high-probability convergence analysis for RS-SGD.}
  Our analysis extends prior high-probability convergence results for \emph{full-dimensional} SGD (without subspace projection) in \cite{liu2023hp} to the randomized subspace setting, yielding bounds with an explicit logarithmic dependence on $\delta$.
  We complement existing in-expectation analyses of RS-SGD by establishing a high-probability guarantee with the same $\epsilon$-dependence in oracle complexity for reaching an $\epsilon$-stationary point.

  \item \textbf{Convergence guarantees for RS-NSGD (in expectation and high probability).}
  For the proposed \emph{randomized subspace normalized SGD} (RS-NSGD), we derive convergence guarantees in expectation with sharp dependence on the matrix-smoothness parameter.
  In addition, we also establish high-probability convergence guarantees. Moreover, our bounds clarify a theoretical basis for when RS-NSGD can yield smaller oracle complexity than the full-dimensional method (NSGD), providing theoretical support for using randomized subspace projection.

\end{itemize}
In addition, \Cref{app:rngd} provides convergence analysis, both in expectation and with high probability, for randomized subspace normalized gradient descent (RS-NGD) in the deterministic setting, where $P_k^\top \nabla F(x_k)$ are used instead of  $P_k^\top g_k$ as projected gradients.

\paragraph{Notation: }
We use the following notation throughout the paper.
Unless stated otherwise, $\|\cdot\|$ denotes the Euclidean norm for vectors
  and the corresponding operator norm for matrices induced by the Euclidean norm. Let $I_d$ denote the 
$d \times d$ identity matrix, and let
   $O(d)$ denote the orthogonal group consisting of all
  $d \times d$ real matrices $U$ satisfying $U^\top U = I_d$.

 \section{Preliminaries}
\subsection{Oracle complexity }
\label{sec:oracle-complexity}

We measure oracle complexity in terms of calls to a coordinate-wise gradient oracle.
Under the oracle model considered here (compatible with forward-mode AD), a \emph{full-dimensional gradient evaluation}
costs $d$ coordinate-wise oracle calls.
In contrast, a \emph{projected gradient evaluation} computes a projection of the gradient onto a
lower-dimensional subspace, e.g., $P^\top \nabla f(x;\xi)$ for a matrix $P\in\mathbb{R}^{d\times r}$ with $r\leq d$,
and thus costs $r$ coordinate-wise oracle calls.
With a minibatch of size $\bar{B}$, these costs scale to $\bar{B}d$ and $\bar{B}r$ oracle calls per iteration, respectively.
This metric serves as a proxy for computational cost in our forward-mode oracle model. 
Computing $P^\top \nabla f(x;\xi)$ requires $r$ directional-derivative evaluations along the columns of $P$; 
thus, the cost scales with $r$ rather than $d$. 
It can also serve as a proxy for communication cost in distributed settings, since communicating an $r$-dimensional sketch 
reduces the number of transmitted scalars from $d$ to $r$. 
Our oracle model targets forward-mode and/or communication-bottlenecked regimes in which the overhead of sampling $P_k$  is negligible compared to directional-derivative evaluations and/or communication\footnote{
In stochastic optimization, \emph{sample complexity} is often used to mean the number of samples accessed by the algorithm
(e.g., $\bar{B}T$ for minibatch size $\bar{B}$ and $T$ iterations). Thus, the oracle complexity equals the sample complexity times the number of coordinate-wise gradient oracle queries per sample access.
}.

\subsection{Properties of Haar Random Matrices}
\label{sec:haar-prop}

Following \citet{kozak2021}, we construct a random matrix $P\in\mathbb{R}^{d\times r}$ using a Haar-random orthogonal matrix, because it satisfies the following four desirable properties: \\
{\bf (i)} For any $U\in O(d)$, $UP$  has the same distribution as $P$.\\
{\bf (ii)} $\mathbb{E}[PP^\top]=I_d$.\\
{\bf (iii)} $P^\top P=\frac{d}{r}I_r$, and hence $\|P\|=\sqrt{\frac{d}{r}}$.\\
{\bf (iv)} For any fixed  $x\in\mathbb{R}^d\setminus\{0\}$, $\|P^\top x\|>0$ almost surely.

\subsection{Problem Setting and Assumptions}

We impose the following standard assumptions on $F$.
\begin{assumption}[Lower Boundedness]\label{ass:lower}
$F$ is bounded below, i.e., $F_\ast := \inf_{x \in \mathbb{R}^d} F(x)> -\infty$; hence, we can define $\Delta_0 \coloneqq F(x_0) - F_\ast$.
\end{assumption}

\begin{assumption}[L-smoothness]\label{ass:Lsmooth}
$F$ is $L$-smooth, i.e., $\exists L >0,\forall x,y\in\mathbb{R}^d$,
\[
\|\nabla F(x) - \nabla F(y)\| \le L \|x - y\|.
\]

Furthermore, we define a non-zero symmetric positive semidefinite matrix 
$\mathbb{L} \in \mathbb{R}^{d \times d}$ such that for all $x,y \in \mathbb{R}^d$,
\begin{align}\label{onesided-Lplus}
    F(y) - F(x) - \langle \nabla F(x),\, y - x \rangle 
    \le \frac{1}{2}\,(y - x)^\top \mathbb{L} (y - x).
\end{align}
\end{assumption}
It is widely known and commonly used in the analysis of non-convex
optimization methods that one can take $\mathbb{L} = L I_d$.\footnote{
However, such choices can be
loose. For example, consider
$F : \mathbb{R}^d \to \mathbb{R}$ defined by
$F(x_1,\dots,x_d) = \frac{x_1^2}{2}.$
This function is $1$-smooth, yet we can choose
$\mathbb{L} = \operatorname{diag}(1,0,\dots,0),$
which yields a much tighter bound.}
Therefore, one may always choose
$\|\mathbb{L}\|\le L$.

This refined smoothness notion is particularly useful for \emph{in-expectation} convergence in our Haar-based analysis; it leads to tighter bounds.

We impose the following unbiasedness assumption on the stochastic gradient.
\begin{assumption}[Unbiasedness]\label{ass:unbiased}
For all $x \in \mathbb{R}^d$,
\[
\mathbb{E}[\nabla f(x,\xi)] = \nabla F(x).
\]
\end{assumption}
We now introduce a few assumptions regarding the noise. 
\begin{assumption}[sub-Gaussian]\label{ass:subgauss}
The gradient satisfies a $\sigma$-sub-Gaussian tail condition, i.e., 
for all $x\in\mathbb{R}^d$ and
all $\lambda\in\mathbb{R}$ with $|\lambda|\le\frac{1}{\sigma}$,
\begin{align}\label{eq:subGaussGrad}
    \mathbb{E}\!\left[
        \exp\!\left(
            \lambda^2\,\|\,\nabla f(x, \xi) -\nabla F(x)\,\|^2
        \right)
    \right]
    \le
    \exp\!\left(
        \lambda^2\sigma^2
    \right).
\end{align}
\end{assumption}
\begin{assumption}[BV]\label{ass:BV}
The gradient has bounded variance, i.e., for all $x \in\mathbb{R}^d$,
\[
\mathbb{E}\bigl[\|\nabla f(x, \xi) - \nabla F(x)\|^2\bigr] \le \sigma^2.
\]
\end{assumption}
While \Cref{ass:BV} is standard, it can be too restrictive in regimes where the stochastic gradient noise is heavy-tailed.
To accommodate such behavior, we allow the noise to have only a bounded $p$-th moment for some $p\in(1,2]$.

\begin{assumption}[$p$-BCM]\label{ass:pBCM}
For some $p\in(1,2]$, the gradient has bounded $p$-th moment, i.e., for all $x\in\mathbb{R}^d$,
\[
\mathbb{E}[\|\nabla f(x, \xi) - \nabla F(x)\|^p] \le \sigma^p.
\]
\end{assumption}

These noise assumptions are related as follows:\footnote{
\textbf{Reason.}
\Cref{ass:subgauss} implies \Cref{ass:BV} by Jensen's inequality.
\Cref{ass:BV} is \Cref{ass:pBCM} with $p=2$, and \Cref{ass:pBCM} for some $p$ implies \Cref{ass:pBCM} for any $q\le p$ (again by Jensen).
Hence, \Cref{ass:BV} implies \Cref{ass:pBCM} for every $p\in(1,2]$.
}
\[
\text{\Cref{ass:subgauss}}
\;\subset\;
\text{\Cref{ass:BV}}
\;\subset\;
\text{\Cref{ass:pBCM}}.
\]

\section{Randomized Subspace Stochastic Optimization}
\label{sec:algorithm}
For a mini-batch stochastic gradient:
\begin{align*} 
\textstyle
g_k \coloneqq \frac{1}{\bar{B}} \sum_{j=1}^{\bar{B}} \nabla f(x_k, \xi_k^j) \in \mathbb{R}^d,
\end{align*}
where $\{\xi_k^j\}_{j=1}^{\bar{B}}$ are i.i.d.\ samples 
from $\mathcal{D}$ at iteration $k$, we construct
 a randomized subspace stochastic gradient as
\[
\textstyle
P_k^\top g_k, ~~~ { i.e., } ~~~
\frac{1}{\bar{B}} \sum_{j=1}^{\bar{B}} P_k^\top \nabla f(x_k, \xi_k^j)  \in \mathbb{R}^r
\]
using a random projection matrix $P_k \in \mathbb{R}^{d\times r}$ constructed from a Haar-random orthogonal matrix.
Here, we propose a general framework for randomized subspace SGD methods using $P_k^\top g_k$ as in \Cref{alg:rs-nsgd}, including the randomized subspace stochastic gradient descent method (RS-SGD) analyzed by \citet{flynn2024}.

\begin{algorithm}[t]
\caption{Randomized Subspace Stochastic Optimization with Scaling Function $\phi$}
\label{alg:rs-nsgd}
\begin{algorithmic}[1]
\REQUIRE Initial point $x_0\in\mathbb{R}^d$, stepsize $\bar{\eta}$, minibatch size $\bar{B}$, subspace dimension $r$, scaling function $\phi$ 
\FOR{$k=0,1,\dots,T-1$}
    \STATE Construct a Haar-random $P_k\in\mathbb{R}^{d\times r}$ and draw i.i.d.\ samples $\{\xi_k^j\}_{j=1}^{\bar{B}} \sim \mathcal{D}$.
 
    \STATE $u_k \gets \frac{1}{\bar{B}} \sum_{j=1}^{\bar{B}} P_k^\top \nabla f(x_k,\xi_k^j)$  \alglinelabel{ln:reduced_grad}
    \STATE $x_{k+1} \gets x_k - \bar{\eta}\phi(u_k) P_k u_k$  \alglinelabel{ln:update}
\ENDFOR
\end{algorithmic}
\end{algorithm}

The computation of the projected gradient $P_k^\top g_k$ in Line~\ref{ln:reduced_grad} can be carried out in different ways, depending on the intended use of the algorithm. When randomized subspace updates are employed as a gradient compression technique in distributed optimization, one may first compute $g_k \in \mathbb{R}^d$ using reverse-mode AD and then project it onto the random subspace via multiplication by $P_k^\top$. In contrast, when the goal is to reduce memory consumption (such as in large-scale models where storing intermediate activations for reverse-mode AD is prohibitive), it is natural to compute $P_k^\top \nabla f(x,\xi_k^j) \in \mathbb{R}^r$ directly using forward-mode AD by evaluating directional derivatives along the columns of $P_k$. In this case, the computational cost scales proportionally to the subspace dimension and amounts to an $r/d$ fraction of the cost of a full-dimensional gradient computation.

\paragraph{RS-SGD ($\phi(u_k)=1$)}
The update rule in Line~\ref{ln:update} becomes
\begin{align}
x_{k+1} = x_k - \bar{\eta} P_k P_k^\top g_k.
\label{eq:update_RS-SGD}
\end{align}
Note that if we choose $r = d$, then $P_k$ becomes an orthogonal matrix
and RS-SGD reduces to standard SGD. Therefore, RS-SGD can be regarded as a generalization of SGD.

While prior work \cite{flynn2024} establishes convergence guarantees for RS-SGD in expectation under a bounded-variance assumption, we show in the next section that RS-SGD also admits high-probability guarantees under sub-Gaussian noise.
However, these noise conditions can be overly restrictive in modern machine-learning workloads, where stochastic gradients are often heavy-tailed.
This limitation reflects the well-known sensitivity of SGD-type updates to heavy-tailed noise. To overcome this issue, we propose a randomized subspace \emph{normalized} stochastic gradient methods (RS-NSGD). 

\paragraph{RS-NSGD($\phi(u_k)=\frac{1}{\|u_k\|}$)}\footnote{We define $\phi(0)=0$}
The update rule in Line~\ref{ln:update} is
\begin{align*}
x_{k+1} = x_k - \bar{\eta} \frac{P_k P_k^\top g_k}{\|P_k^\top g_k\|}.
\end{align*}
As in the discussion of RS-SGD, when 
$r = d$, RS-NSGD reduces to the standard NSGD; hence, RS-NSGD can be viewed as a generalization of NSGD.

\section{High-Probability Convergence of RS-SGD}

As mentioned earlier, \citet{flynn2024} established an in-expectation convergence guarantee for RS-SGD.
Here we prove a nonasymptotic high-probability guarantee in the setting of Algorithm~1.
Our analysis extends the framework of \citet{liu2023hp} to mini-batch gradients.
Proofs in this section are deferred to \Cref{app:rsgd-hp}.

\begin{lemma}\label{lem:subgauss_minibatch}
Assume that the stochastic gradient  $\nabla f(x,\xi)$ satisfies
\Cref{ass:unbiased,ass:subgauss}.
For a fixed $x \in \mathbb{R}^d$, define the mini-batch gradient estimator
\[
g(x) \coloneqq \frac{1}{\bar{B}} \sum_{j=1}^{\bar{B}} \nabla f(x,\xi^j),
\]
and set
\[
\bar{\sigma} \coloneqq 
24\sqrt{\frac{e(e+1)}{2}}\frac{\sigma}{\sqrt{\bar{B}}}.
\]
Then for all $\lambda$ satisfying $|\lambda| \le 1/\bar{\sigma}$, it holds that
\[
\mathbb{E}\!\left[
\exp\!\left(\lambda^2 \bigl\|g(x) - \nabla F(x)\bigr\|^2\right)
\right]
\le
5^d \exp\!\left(\lambda^2 \bar{\sigma}^2\right).
\]
\end{lemma}
By \Cref{lem:subgauss_minibatch}, increasing the mini-batch size reduces the noise, and 
the following theorem establishes convergence.

\begin{theorem}\label{thm:main_highprob}
Suppose that \Cref{ass:lower,ass:Lsmooth,ass:unbiased,ass:subgauss} hold. 
Let $\{x_k\}_{k\ge 0}$ be the sequence generated by  RS-SGD.
Let $\bar{B}\coloneqq \lceil \max\{1, BT^q\} \rceil$ for $B>0$ and $q>0$,
$\alpha \coloneqq 24\sqrt{\frac{e(e+1)}{2}}$,
and $\bar{\eta} = \frac{r}{Ld}$.
Further define
\begin{equation}
\mu \coloneqq
\begin{cases}
1 - I_{\frac{r}{2d}}\!\left(\frac{r}{2}, \frac{d-r}{2}\right), & r<d,\\
1, & r=d.
\end{cases}
\label{eq:def_mu}
\end{equation}
where $I_{x}(a,b)$ denotes the regularized incomplete Beta function.
Then for any $\delta \in (0,1)$ and any $T$ such that
$2\log\!\left(\frac{2}{\delta}\right) < \mu T$, the following holds 
with probability at least $1-\delta$: 
\begin{align*}
\min_{0\leq k \leq T-1}\|\nabla F(x_k)\|^2 
\leq 
\frac{
4L\frac{d}{r}\Delta_0
}{
\mu T - \sqrt{2\mu T \log\!\left(\frac{2}{\delta}\right)}
}
\\
+
\frac{
2\frac{d}{r}\left(\frac{\alpha^2\sigma^2}{\bar{B}}\right)
\bigl(d\log 5 + 1\bigr)
}{
\mu - \sqrt{2\frac{\mu}{T}\log\!\left(\frac{2}{\delta}\right)}
}
+
\frac{
2\frac{d}{r}\left(\frac{\alpha^2\sigma^2}{\bar{B}}\right)
\log\!\left(\frac{2}{\delta}\right)
}{
\mu T - \sqrt{2\mu T \log\!\left(\frac{2}{\delta}\right)}
}.
\end{align*}
\end{theorem}

For $\mu$ defined in \eqref{eq:def_mu}, for all $1 \le r \le d$,
\footnote{
If $r=d$, then $\mu=1$; otherwise let $Z\sim\mathrm{Beta}(\frac r2,\frac{d-r}{2})$.
By definition of $I_x(a,b)$ and $\mathbb{E}[Z]=r/d$,
$\mu=\mathbb{P}(Z\ge r/(2d))=\mathbb{P}(Z\ge \frac12\,\mathbb{E}[Z])$.
Paley--Zygmund \cite{paley1932} gives
$\mu \ge \frac14\,\frac{\mathbb{E}[Z]^2}{\mathbb{E}[Z^2]}$,
and using $\mathbb{E}[Z^2]=\frac{r(r+2)}{d(d+2)}$ yields $\mu\ge \frac{1}{12}$.
}
$\frac{1}{12} \leq \mu \le 1.$

From the result of \Cref{thm:main_highprob}, we can derive the oracle complexity required to achieve an $\epsilon$-stationary point:
\begin{equation}
\min_{0 \le k \le T-1} \|\nabla F(x_k)\| \le \epsilon 
\label{eqn:eps-stationary}
\end{equation}
Take $q=1$ and $B=\frac{d\sigma^2}{L\Delta_0}$ for $\bar{B}$.
Then the oracle complexity 
is \footnote{Throughout, $\mathcal{O}$ notation hides constants; statements are meant to hold for all sufficiently small $\epsilon$, with the threshold possibly depending on parameters. We use $\tilde{\mathcal{O}}(\cdot)$ to additionally hide polylogarithmic factors in $1/\delta$.}
\[
\tilde{\mathcal{O}}\!\left(
\frac{d^{3}}{\mu^{2}r}\,\Delta_0L\sigma^{2}\,\epsilon^{-4}
\right).
\]
Expectation bounds capture average behavior but not the likelihood of large deviations in one run.
Our high-probability analysis quantifies this risk.
To the best of our knowledge, \Cref{thm:main_highprob} provides the first high-probability convergence guarantee for RS-SGD under the sub-Gaussian noise assumption.
Although the constants become larger compared to in-expectation bounds, our result matches the $\epsilon$-dependence ($\epsilon^{-4}$) implied by the in-expectation analysis of \citet{flynn2024} for achieving \eqref{eqn:eps-stationary}; it also matches the standard $\epsilon^{-4}$ high-probability rate of SGD under sub-Gaussian noise \cite{liu2023hp}.

\section{Convergence Results for RS-NSGD}
Proofs of the results in this section are deferred to \Cref{app:rnsgd}.
First, we obtain the following important proposition.

\begin{proposition}
\label{prop:tau}
Let \(x \neq 0\), and let \(P\) be a Haar matrix.
Then \(\|P^\top x\| \neq 0\) almost surely, and there exists a constant
\(\tau > 0\) such that
\[
\mathbb{E}\!\left[\frac{PP^\top x}{\|P^\top x\|}\right]
=
\tau\,\frac{x}{\|x\|}.
\]
Specifically, \(\tau\) is given by
\[
\tau
=
\sqrt{\frac{d}{r}}\,
\frac{
\Gamma\!\left(\frac{r+1}{2} \right)
\Gamma\!\left(\frac{d}{2}\right)
}{
\Gamma\!\left(\frac{r}{2}\right)
\Gamma\!\left(\frac{d+1}{2} \right)
},
\]
where $\Gamma(\cdot)$ denotes the Gamma function.
Furthermore,  $\frac{1}{\sqrt{2}} \leq \tau \leq 1.$
\end{proposition}

When $r = d$, the coefficient $\tau$ becomes $1$. Moreover, as illustrated in \Cref{fig:tau}, even for small $r$, the value of $\tau$ quickly approaches $1$.

\begin{figure}[t]
    \centering
    \includegraphics[width=0.4\columnwidth]{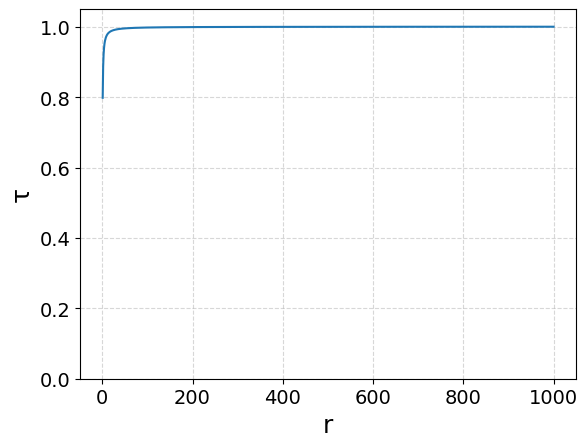}
    \caption{The value of $\tau$ as a function of $r$ for $d = 1000$.}
    \label{fig:tau}
\end{figure}

\subsection{In-expectation Bound}
To obtain tight bounds on the quadratic terms in the in-expectation convergence analysis, we derive the following elementary proposition from properties of the Haar matrix and basic linear algebra.
\begin{proposition}\label{prop:secondterm}
Let $S$ be a $d\times d$ square matrix, and $P$ be a Haar matrix. Then,
\[
\mathbb{E}\left[\frac{v^\top PP^\top S PP^\top v}{\|P^\top v\|^2}\right]
=
\frac{d(r-1)}{r(d-1)}\frac{v^\top Sv}{\|v\|^2}
+\frac{d-r}{r(d-1)}\operatorname{tr}(S).
\]
\end{proposition}
When bounding the first term in \Cref{prop:secondterm} by $\|S\|$,
the relationship between $\operatorname{tr}(S)$ and $\|S\|$ becomes important.
To this end, we introduce the following quantity.

\begin{definition}\label{def:effective-rank}
Let $S$ be a nonzero positive semidefinite matrix. 
We define the effective rank of $S$ as
\[
r_{\mathrm{eff}}(S) := \frac{\operatorname{tr}(S)}{\|S\|}.
\]
Based on this, we introduce the quantity
\[
\ell(S) := \frac{d(r-1) + r_{\mathrm{eff}}(S)(d-r)}{r(d-1)}.
\]
\end{definition}
Since $\ell(S)$ is monotonically increasing with respect to 
$r_{\mathrm{eff}}(S)$ and $r_{\mathrm{eff}}(S) \in [1, d]$, 
we have $\ell(S) \in [1, d/r]$. 
Furthermore, when $r \neq d$, $\ell(S) = d/r$ holds 
if and only if $r_{\mathrm{eff}}(S) = d$, 
which occurs when all eigenvalues of $S$ are equal, that is, when $S$ is a scalar multiple of the identity matrix.
Moreover, when $r = d$, we simply have $\ell(S) = 1$ for all $S$\footnote{Furthermore, even when the smoothness matrix $\mathbb{L}$ is known but the exact values of
$\operatorname{tr}(\mathbb{L})$ and $\|\mathbb{L}\|$ are unavailable,
our analysis still applies as long as upper bounds
$\operatorname{tr}(\mathbb{L}) \le \alpha$ and $\|\mathbb{L}\| \le \beta$
are available.
In this case, set $r_{\mathrm{eff}} \coloneqq \alpha/\beta$, define $\ell$ accordingly, and the rest of the analysis is unchanged.
For example, in binary logistic regression considered in~\cite{flynn2024},
$\alpha$ and $\beta$ are given as $1/4$.
As a result, the value of $\ell$ defined in this manner becomes $1$.}.

We are now ready to state an in-expectation convergence bound for RS-NSGD.
\begin{theorem}\label{Lsmoothrnsgd}
Suppose that \Cref{ass:lower,ass:Lsmooth,ass:unbiased,ass:pBCM} hold.
Let $\{x_k\}_{k\ge 0}$ be the sequence generated by RS-NSGD.
We set
\[
\bar{B} = \left\lceil \max\{1,\, B T^q\} \right\rceil,
\qquad
\bar{\eta} = \eta T^{-u},
\]
for parameters $B > 0$, $q > 0$, $\eta > 0$, and $u \in (0,1)$.
Then it holds that
\begin{align*}
\frac{1}{T} \sum_{k=0}^{T-1} \mathbb{E}[\|\nabla F(x_k)\|]
\le
\frac{\Delta_0}{\tau \eta T^{1-u}}
\\ + \frac{\eta \ell(\mathbb{L})\|\mathbb{L}\|}{2\tau T^u}
+ \frac{4\sigma}{\max\{1,BT^q\}^{(p-1)/p}}.
\end{align*}
\end{theorem}
As noted earlier, when $r = d$, RS-NSGD reduces to NSGD. Indeed, substituting
$r=d$ and $\|\mathbb{L}\|\le L$ into \Cref{Lsmoothrnsgd} yields a convergence bound that exactly matches the known result for NSGD~\citep{huebler2025}.

From \Cref{Lsmoothrnsgd}, we can derive the corresponding oracle complexity bounds.

We first consider the setting where the exact value of $p$ and other problem-parameters are unknown.
In this case, we take $u=\tfrac{1}{2}$ and $q=1$
\footnote{
As in \cite{huebler2025}, when $p$ is unknown in $(1,2]$, no fixed $q\neq 1$ can uniformly improve the $\epsilon$-exponent for all $p\in(1,2]$.
Moreover, when $p=2$, the $\epsilon$-exponent is uniquely minimized at $q=1$. By continuity in $p$, the same  holds when $p$ is unknown in $(1,2)$.
Therefore, even if one only knows that $p<2$, 
choosing $(u,q)=(\tfrac12,1)$ remains a natural default.
}
.
With this choice, taking $\eta$ and $B$ to be constants,
the resulting oracle complexity is given by
\[
\mathcal{O}\!\left(
r\left(\frac{\Delta_0}{\tau \epsilon}\right)^{4}
+
r\left(\frac{\ell(\mathbb{L})\,\|\mathbb{L}\|}{\tau \epsilon}\right)^{4}
+
r\left(\frac{\sigma}{\epsilon}\right)^{\frac{2p}{p-1}}
\right).
\]

For comparison, the corresponding bound for the full-dimensional method is obtained by substituting $r=d$ into the above expression:
\[
\mathcal{O}\!\left(
d\left(\frac{\Delta_0}{\epsilon}\right)^{4}
+
d\left(\frac{\|\mathbb{L}\|}{\epsilon}\right)^{4}
+
d\left(\frac{\sigma}{\epsilon}\right)^{\frac{2p}{p-1}}
\right).
\]
In particular, when $p=2$ (so that $\frac{2p}{p-1}=4$), the RS-NSGD bound can be written as
\[
\mathcal{O}\!\Bigl(
\epsilon^{-4}\, r\Bigl(\tfrac{\Delta_0^{4}}{\tau^{4}}
+\tfrac{(\ell(\mathbb{L})\|\mathbb{L}\|)^{4}}{\tau^{4}}
+\sigma^{4}\Bigr)
\Bigr),
\]
whereas the full-dimensional bound becomes
\[
\mathcal{O}\!\Bigl(
\epsilon^{-4}\, d\bigl(\Delta_0^{4}+\|\mathbb{L}\|^{4}+\sigma^{4}\bigr)
\Bigr).
\]
Thus, for instance when $\ell(\mathbb{L})$ is small, some subspace dimensions $r<d$ can reduce the leading $\epsilon^{-4}$ coefficient and yield a smaller oracle complexity than the full-dimensional method.

When $p<2$, the noise term dominates and the bounds simplify to
\[
\text{RS-NSGD: }\ \mathcal{O}\!\left(r\left(\frac{\sigma}{\epsilon}\right)^{\frac{2p}{p-1}}\right)
\qquad
\text{NSGD: }\ \mathcal{O}\!\left(d\left(\frac{\sigma}{\epsilon}\right)^{\frac{2p}{p-1}}\right),
\]
which shows that RS-NSGD improves the leading dependence by a factor $r/d$.

Next, when $p$ and the other problem-parameters are known, take
$u=\tfrac{1}{2}$ and $q=\tfrac{p}{2(p-1)}$.
By setting
\[
\eta \coloneqq \sqrt{\frac{\Delta_0}{\ell(\mathbb{L})\,\|\mathbb{L}\|}},
\qquad
B \coloneqq
\left(
\frac{\sigma\,\tau}{\sqrt{\Delta_0\ell(\mathbb{L})\|\mathbb{L}\|}}
\right)^{\frac{p}{p-1}},
\]
we obtain the oracle complexity bound
\[
\mathcal{O}\!\left(
r\,\frac{\ell(\mathbb{L})}{\tau^2}
\frac{\Delta_0\,\|\mathbb{L}\|}{\epsilon^{2}}
\left(\frac{\sigma}{\epsilon}\right)^{\frac{p}{p-1}}
\right).
\]

For comparison, the full-dimensional case (NSGD, i.e., $r=d$) is obtained by simply setting $r=d$ in the above bound:
\[
\mathcal{O}\!\left(
d\,
\frac{\Delta_0\,\|\mathbb{L}\|}{\epsilon^{2}}
\left(\frac{\sigma}{\epsilon}\right)^{\frac{p}{p-1}}
\right).
\]
Hence, the RS-NSGD rate is a factor $\frac{r}{d}\frac{\ell(\mathbb{L})}{\tau^2}$ times the NSGD rate. Since $\ell(\mathbb{L}) \le d/r$, this factor satisfies 
$\frac{r}{d}\frac{\ell(\mathbb{L})}{\tau^2}\le 1/\tau^2$; 
moreover, $\tau$ stays close to $1$ even when $r$ is small as shown in \Cref{fig:tau}, so this factor is roughly $1$.
Therefore, except in the extreme case $r_{\mathrm{eff}}(\mathbb{L})=d$, where $\ell(\mathbb{L})$ essentially saturates $d/r$, RS-NSGD can reduce the oracle cost.

\subsection{High-probability Bound}
We now derive the following high-probability bound.
\begin{theorem}\label{high:minibatch-RNSGD-prob}
We assume the same setting with
\Cref{Lsmoothrnsgd} for $F$, the gradient , $\bar{B}$ and $\bar{\eta}$ using   
parameters $B > 0$, $q > 0$, $\eta > 0$, and $u \in (0,1)$.
Then, for the sequence $\{x_k\}_{k\ge 0}$ of RS-NSGD,
the following holds with probability at least $1 - \delta$:
\begin{align*}
\frac{1}{T} \sum_{k=0}^{T-1} \|\nabla F(x_k)\|
\le \frac{2 \Delta_0}{\tau \eta T^{1-u}}
+ 17\frac{\frac{d}{r} \sqrt{\Delta_0 L}}{\tau^2 T} \log (1/\delta)
\\
+ \frac{\frac{d}{r} \eta L}{\tau T^u} \big( 1 + 12\frac{\sqrt{\frac{d}{r}}}{\tau} \log (1/\delta) \big)
+ \frac{8 \sigma}{\max \{1, B T^q\}^{\frac{p-1}{p}}}.
\end{align*}
\end{theorem}
Similarly to \Cref{Lsmoothrnsgd}, substituting $r = d$ into \Cref{high:minibatch-RNSGD-prob} yields a bound that exactly matches the known result for NSGD~\citep{huebler2025}.

From \Cref{high:minibatch-RNSGD-prob}, we can  derive the oracle complexity.
When the problem-parameters (including $p$) are unknown, we adopt, as a natural default, the same choice used in the in-expectation setting and set
$u=\tfrac{1}{2}$ and $q=1$.
Taking $\eta$ and $B$ to be constants,
the resulting oracle complexity is given by
\[
\tilde{\mathcal{O}}\!\left(
r\left(\frac{\Delta_0}{\tau\epsilon}\right)^4
+
\frac{d^6}{r^5}
\left(\frac{L}{\tau^2 \epsilon}\right)^4
+
r\left(\frac{\sigma}{\epsilon}\right)^{\frac{2p}{p-1}}
\right).
\]
For comparison, the full-dimensional method corresponds to NSGD (i.e., $r=d$):
\[
\tilde{\mathcal{O}}\!\left(
d\left(\frac{\Delta_0}{\epsilon}\right)^4
+
d\left(\frac{L}{\epsilon}\right)^4
+
d\left(\frac{\sigma}{\epsilon}\right)^{\frac{2p}{p-1}}
\right).
\]
In particular, when $p=2$, both bounds scale as $\epsilon^{-4}$, and RS-NSGD can yield a smaller oracle complexity than NSGD, depending on the problem constants, for some $r<d$.
When $p<2$, similarly to the in-expectation case, RS-NSGD improves the leading term by a factor $r/d$ compared to NSGD.

Next, when the value of $p$ as well as the other problem-parameters are known,
we adopt the same choice used in the in-expectation setting and take
$u=\tfrac{1}{2}$ and $q=\tfrac{p}{2(p-1)}$.
Set
\[
\eta \coloneqq
\left(\frac{r}{d}\right)^{\frac{3}{4}}
\sqrt{\frac{\Delta_0 \tau}{L}},
\qquad
B \coloneqq
\left(
\frac{\sigma \tau^{\frac{3}{2}}}{\sqrt{\Delta_0 L}}
\left(\frac{r}{d}\right)^{\frac{3}{4}}
\right)^{\frac{p}{p-1}}.
\]
Then 
the oracle complexity becomes
\[
\tilde{\mathcal{O}}\!\left(
\left(\frac{d}{r^{1/3}\tau^2}\right)^{3/2}
\frac{\Delta_0 L}{\epsilon^2}
\left(\frac{\sigma}{\epsilon}\right)^{\frac{p}{p-1}}
\right).
\]
In this case, the bound is minimized by the choice $r=d$ (i.e., NSGD).

Overall, these results yield a theoretical guideline for when to use randomized subspace projection:
the answer depends on the available problem information—including whether $p$ is unknown, only known to satisfy $p\neq 2$, or known—and on whether one needs in-expectation or high-probability guarantees.

\section{Experiments}

\begin{figure*}[t]
  \centering
  \begin{subfigure}[t]{0.32\textwidth}
    \centering
    \includegraphics[width=\linewidth]{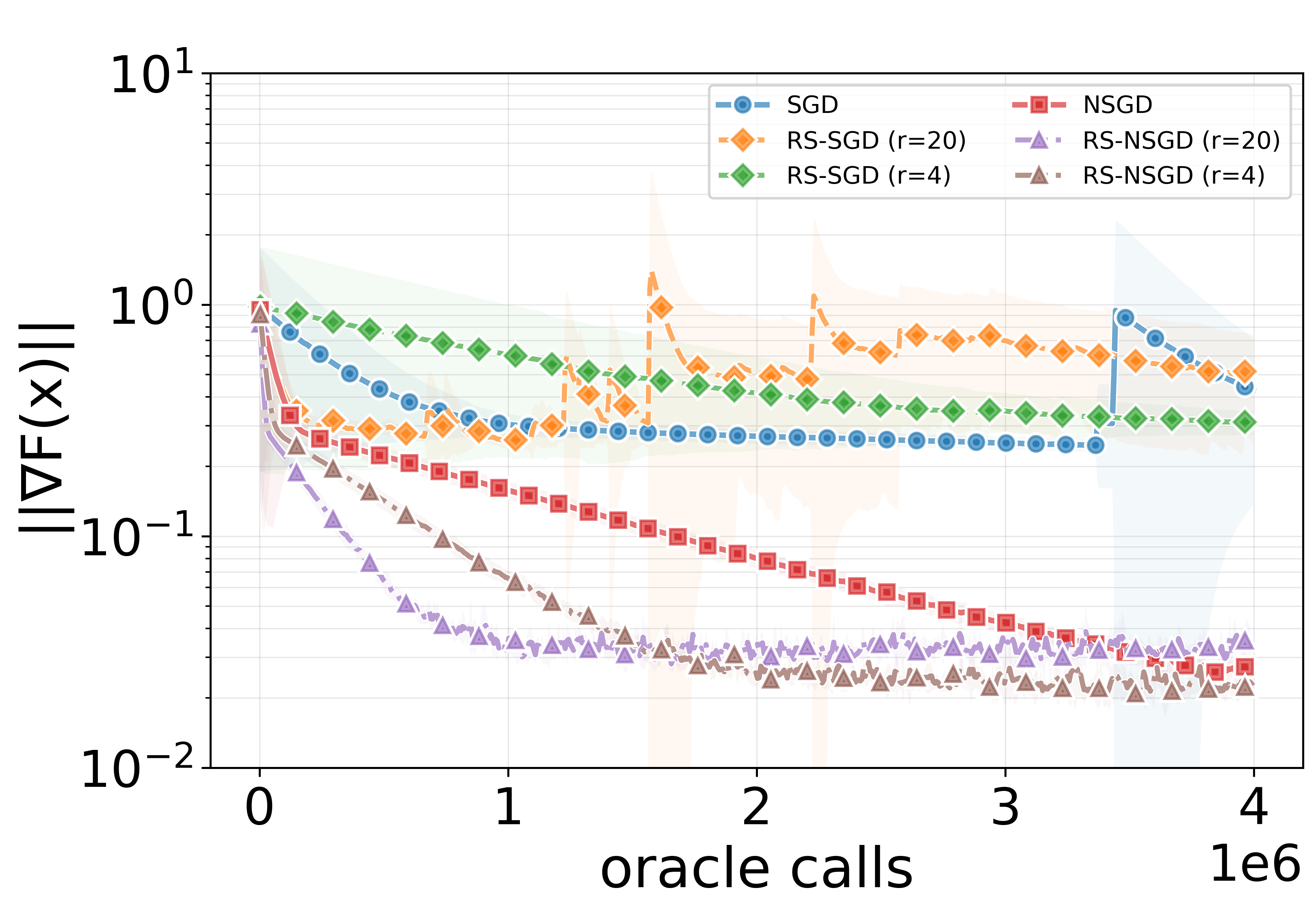}
    \caption{$\rho=4$}
    \label{fig:quad-v4}
  \end{subfigure}\hfill
  \begin{subfigure}[t]{0.32\textwidth}
    \centering
    \includegraphics[width=\linewidth]{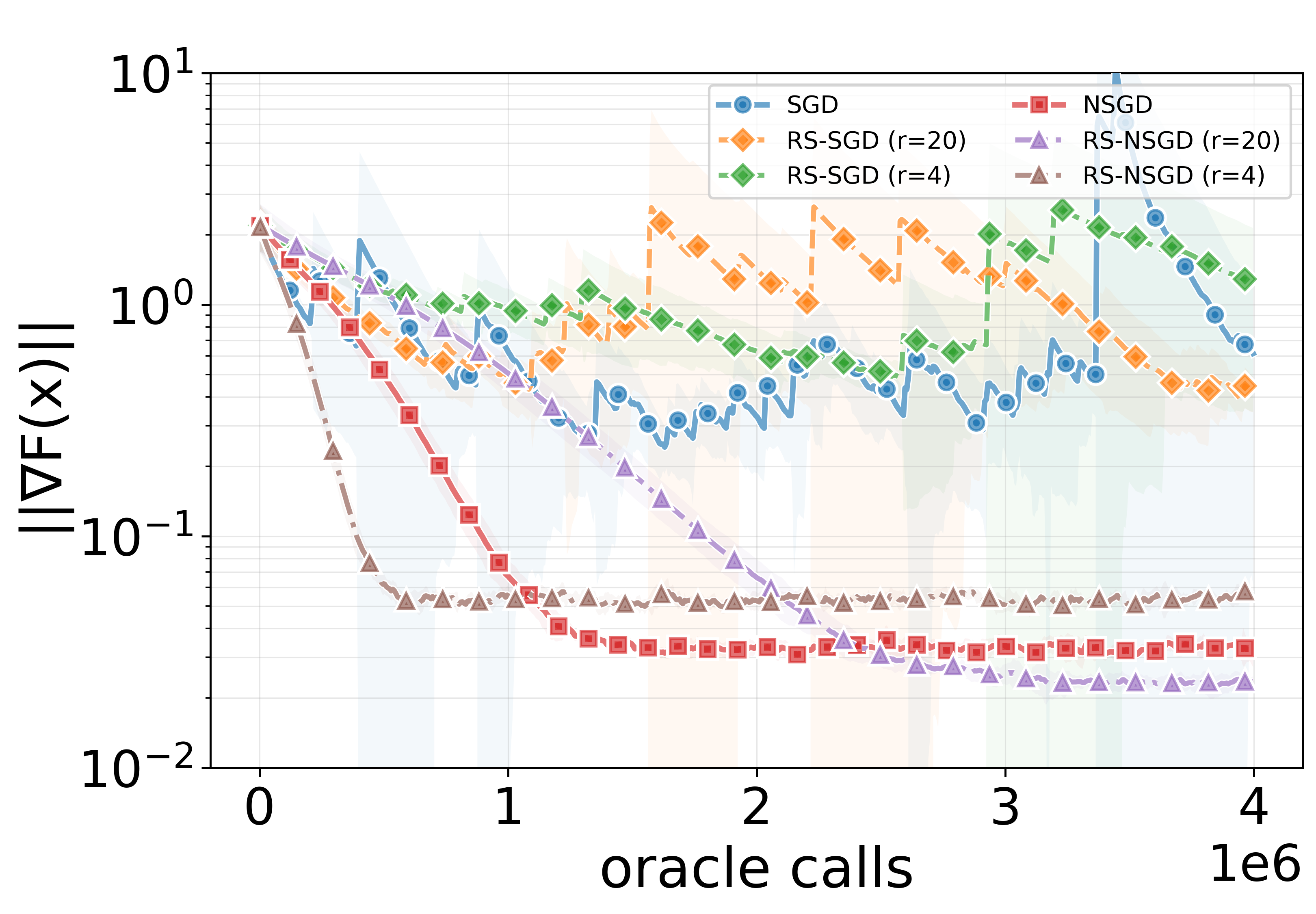}
    \caption{$\rho=20$}
    \label{fig:quad-v20}
  \end{subfigure}\hfill
  \begin{subfigure}[t]{0.32\textwidth}
    \centering
    \includegraphics[width=\linewidth]{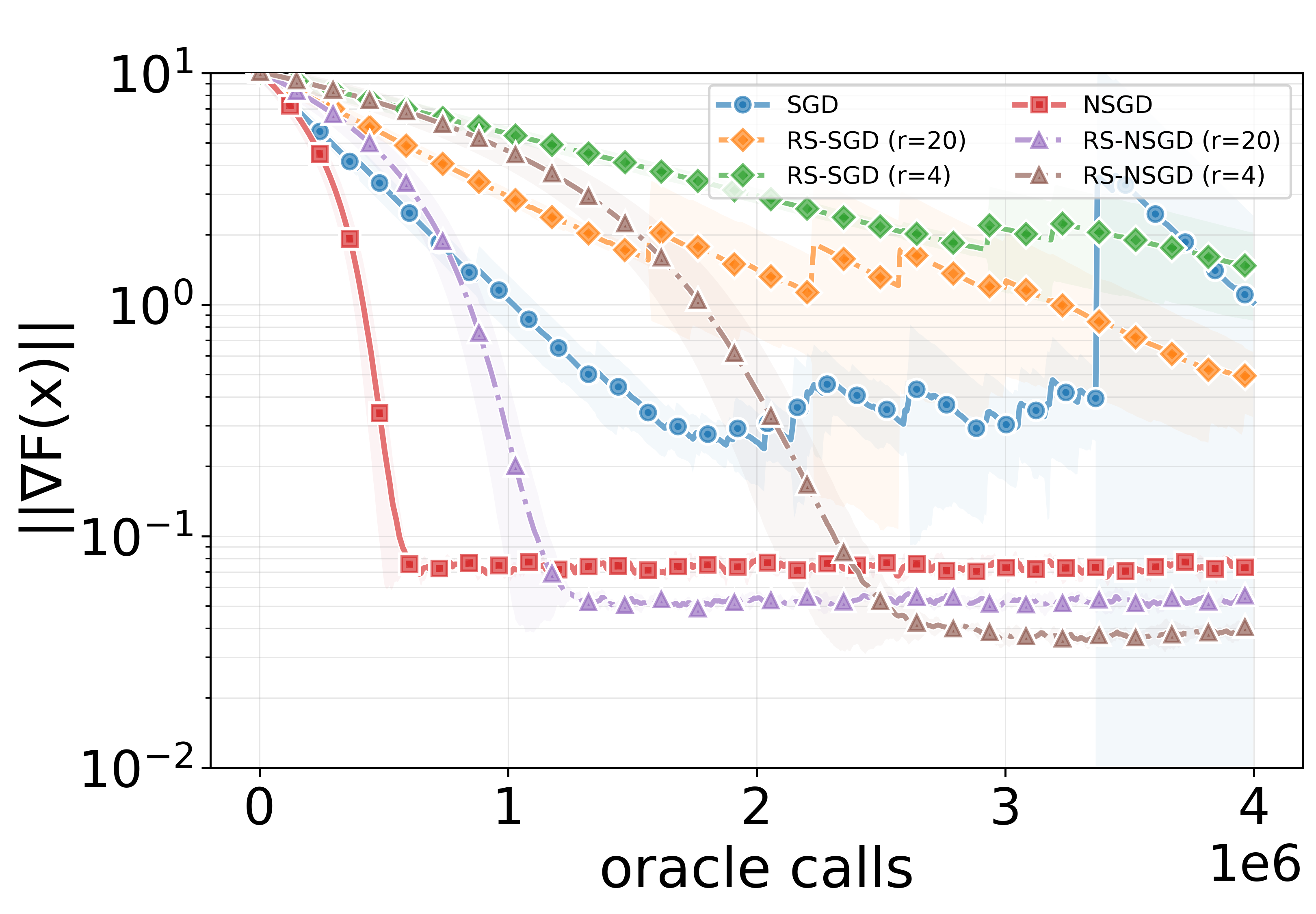}
    \caption{$\rho=100$}
    \label{fig:quad-v100}
  \end{subfigure}
  \caption{Synthetic quadratic with heavy-tailed noise ($d=100$, $\bar{B}=4$): oracle calls vs.\ $\|\nabla F(x)\|=\|\Lambda x\|$ (mean $\pm$ 1 std., 5 seeds).}
  \label{fig:quad-v-all}
\end{figure*}

\begin{figure*}[t]
  \centering
  \scalebox{0.65}{
  \begin{minipage}[t]{0.49\textwidth}
    \centering
    \includegraphics[width=\linewidth]{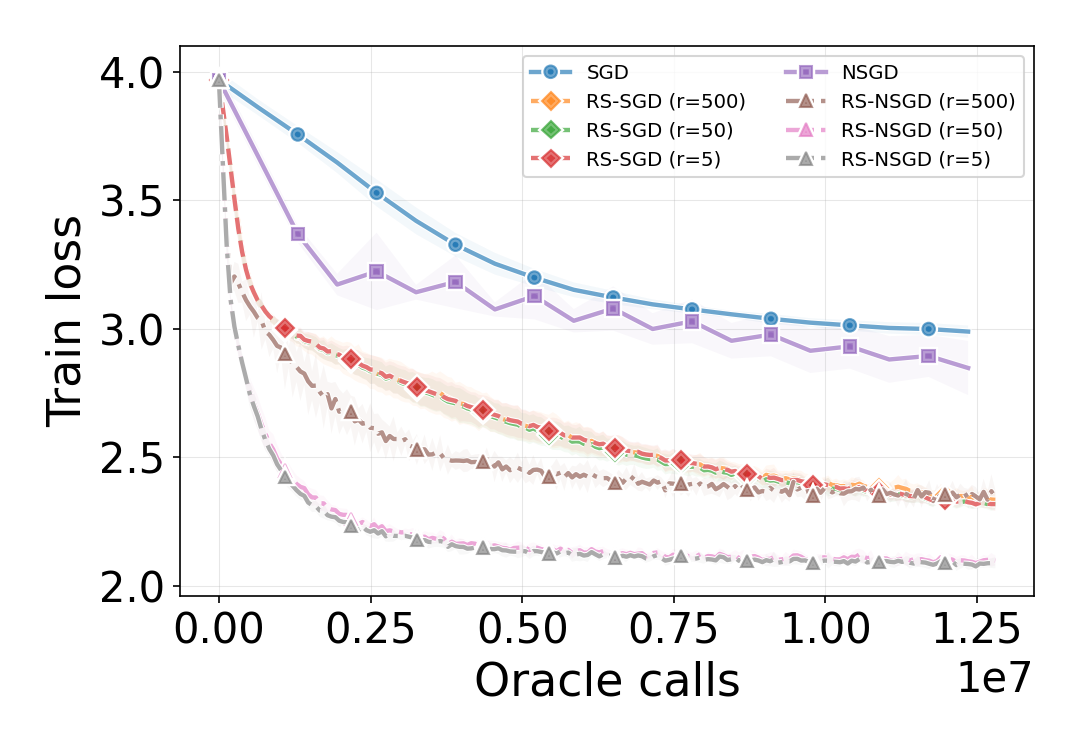}
    {\small\textbf{(a)} PTB}
  \end{minipage}\hfill
  \begin{minipage}[t]{0.49\textwidth}
    \centering
    \includegraphics[width=\linewidth]{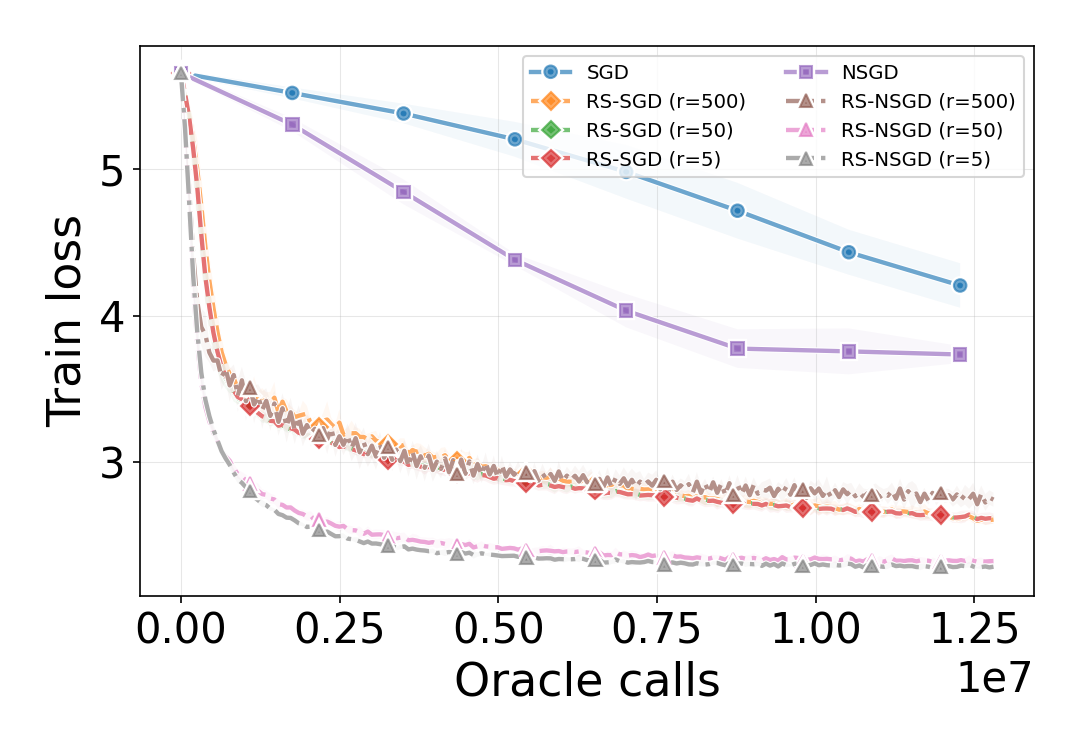}
    {\small\textbf{(b)} WikiText-2}
  \end{minipage}
  }
  \caption{Character-level language modeling on PTB ($d=5079,\bar{B}=128$) and WikiText-2 ($d=13700,\bar{B}=128$): oracle calls vs.\ training loss (mean $\pm$ 1 std., 5 seeds).}

  \label{fig:charlm-trainloss}
\end{figure*}
The purpose of this section is to empirically validate the behavior suggested by our theoretical analysis.
We complement our theoretical results with numerical experiments on both controlled synthetic problems and real datasets.
Further experimental details (including tuning details and additional results) are provided in \Cref{app:exp_details}.

\subsection{Synthetic quadratic with heavy-tailed noise}
\label{sec:exp-synth}
We consider a stochastic quadratic objective in dimension $d=100$,
\begin{equation}
f(x;\xi)\coloneqq \tfrac12 x^\top \Lambda x + \langle \xi, x\rangle,
\label{eq:synth-obj}
\end{equation}
where $\xi\in\mathbb{R}^d$ is sampled from a \emph{symmetrized Pareto} distribution with tail index
$\hat{p}=1.2$. In particular, $\xi$ has a bounded $p$-th moment for any $p<\hat{p}$.
We set $\Lambda=\diag(\lambda_1,\ldots,\lambda_d)$ as follows:
\begin{equation}
\Lambda=\diag\!\left(1,\frac{\rho-1}{d-1},\ldots,\frac{\rho-1}{d-1}\right),
\label{eq:spectrum}
\end{equation}
where $\rho\le d$.
This choice ensures that $r_{\mathrm{eff}}(\Lambda)=\mathrm{tr}(\Lambda)/\|\Lambda\|=\rho$.
We report results for $\rho\in\{4,20,100\}$.

We compare six methods: SGD, NSGD, RS-SGD with $r\in\{4,20\}$, and RS-NSGD with $r\in\{4,20\}$, using minibatch size $\bar{B}=4$.
We tune the stepsize $\bar{\eta}$ for each method and report trajectories under a fixed oracle-call budget; details are given in \Cref{app:exp_details}.

\Cref{fig:quad-v-all} summarize the results for $\rho\in\{4,20,100\}$, plotting the number of oracle calls (x-axis) versus $\|\nabla F(x)\|$ (y-axis).
Each curve reports the mean over $5$ evaluation seeds, with a shaded band indicating $\pm 1$ standard deviation.
Note that NSGD can be interpreted as a special case of RS-NSGD with $r=d=100$ (and similarly, SGD is RS-SGD with $r=d=100$).
Across all settings, SGD and RS-SGD ($r\in\{4,20\}$) exhibit limited progress under heavy-tailed noise and often fail to consistently decrease $\|\nabla F(x)\|$. Moreover, the wider standard-deviation bands of SGD and RS-SGD suggest less stable optimization behaviour, plausibly due to their stronger sensitivity to heavy-tailed noise.
In contrast, NSGD and RS-NSGD ($r\in\{4,20\}$) show a clear decrease of $\|\nabla F(x)\|$.
Moreover, when the effective rank is largest ($r_{\mathrm{eff}}(\Lambda)=\rho=100$), NSGD achieves the fastest reduction per oracle call, whereas for smaller effective ranks ($\rho\in\{4,20\}$), RS-NSGD is more oracle-efficient than NSGD. This behavior is compatible with our theory.

\subsection{Character-level language modeling}
\label{sec:exp-charlm}
Since stochastic gradients can be heavy-tailed in language modeling \citep{zhang2020clipped}, we use language modeling as a real-data testbed. We evaluate on the character-level Penn Treebank (PTB) \cite{marcus1993ptb} and WikiText-2 (v1) \cite{merity2017pointer} benchmarks.
We train a single-layer GRU language model \cite{cho2014gru} with embedding dimension $E=8$ and hidden size $H=28$.

We train a character-level language model. The character vocabulary is built from the training split (with an \texttt{<unk>} token for unseen characters in validation/test).
At each optimization step, we sample $\bar{B}=128$ length-$35$ subsequences i.i.d. from the training corpus (by choosing start positions uniformly with replacement).
The resulting parameter dimension is $d=5079$ on PTB and $d=13700$ on WikiText-2.
For plotting, we compute the training loss as the average mini-batch loss over five independent mini-batches.
We consider subspace dimensions $r\in\{5,50,500\}$ for RS-SGD/RS-NSGD.
We tune the stepsize $\bar{\eta}$ for each method; the tuning and evaluation protocol is described in \Cref{app:exp_details}.

\Cref{fig:charlm-trainloss} reports oracle calls versus training loss on PTB and WikiText-2.
Each curve shows the mean over five seeds, with shaded regions indicating $\pm 1$ standard deviation.
Across both datasets, RS-NSGD with $r\in\{5,50\}$ achieves the strongest performance under the oracle budget.

\section{Conclusion}
We proved high-probability convergence for randomized subspace SGD under sub-Gaussian noise, complementing prior expectation-based results. We also proposed RS-NSGD and established in-expectation and high-probability guarantees under a bounded $p$-th moment assumption, characterizing when it can improve oracle complexity over NSGD. Our guarantees rely on Haar-random subspace properties (\Cref{sec:haar-prop}). Extending them to other constructions (e.g., Gaussian or block-coordinate subspaces) is nontrivial; for instance, Gaussian matrices are unbounded, and block-coordinate schemes can satisfy $\mathbb{P}(P^\top x=0)>0$ for some fixed $x\neq 0$. Establishing analogous guarantees by replacing the Haar-specific ingredients is an interesting direction for future work.

\section*{Acknowledgements}
This work was partially supported by JSPS Grant-in-Aid for Scientific Research(B) JP23H03351 and JST CREST Grant Number JPMJCR24Q2.
\bibliographystyle{plainnat}
\bibliography{references}

\onecolumn
\appendix

\section{Related Work}
\label{app:related-work}
\subsection{Optimization via Random Subspaces}
In many large-scale optimization problems arising in PDE-constrained optimization and machine learning, the memory required to store intermediate variables for reverse-mode automatic differentiation becomes prohibitive. 
In such settings, forward-mode automatic differentiation becomes more attractive due to its significantly lower memory footprint.
This issue is particularly pronounced in modern machine learning models, such as deep neural networks
and large transformer architectures, where the activation memory dominates the memory usage.
 However, naively computing $d$ directional derivatives in a $d$-dimensional space is computationally very expensive. 
One possible way to reduce computation is to work in random low-dimensional subspaces: Using  a random matrix $P_k \in \mathbb{R}^{d \times r}$ with $r \ll d$, one approximates the full gradient $\nabla F(x_k)$ via directional derivatives along the columns of $P_k$ and performs the update
\[
  x_{k+1} = x_k - \bar{\eta} P_k P_k^\top \nabla F(x_k).
\]
This subspace-descent approach was developed for deterministic optimization 
by Kozak et al.~\cite{kozak2021}.

Gradient compression techniques have been widely studied to alleviate communication bottlenecks in distributed SGD.
Classical approaches include quantization or error-feedback variants, along with low-rank projection methods such as
PowerSGD~\cite{seide20141bit,karimireddy2019errorfeedback,vogels2019powersgd}. More recently, randomized linear compressors—including those based on Haar matrices—have been shown to provide convergence in non-convex optimization~\cite{flynn2024}. This provides theoretical support for using Haar matrices as randomized compressors in optimization.
Rather than communicating the full stochastic gradient $g_k\in\mathbb{R}^d$ at the current iterate,
we draw a random $r$-dimensional subspace represented by a matrix $P_k\in\mathbb{R}^{d\times r}$
(e.g., with orthonormal columns) and perform the projected update
\[
x_{k+1}=x_k-\bar{\eta}\,P_kP_k^\top g_k
      \;=\; x_k-\bar{\eta}\,P_k\bigl(P_k^\top g_k\bigr).
\]
In a minibatch setting where
\[
g_k \;=\; \frac{1}{\bar{B}}\sum_{j=1}^{\bar{B}} \nabla f(x_k,\xi_k^j),
\]
the communicated quantity is the projected per-sample (or per-worker) gradient
\[
u_k^j \;:=\; P_k^\top \nabla f(x_k,\xi_k^j)\in\mathbb{R}^r,
\]
and the aggregator forms
\[
P_k^\top g_k \;=\; \frac{1}{\bar{B}}\sum_{j=1}^{\bar{B}} u_k^j.
\]
Thus, each worker transmits an $r$-dimensional vector per iteration (instead of $d$ dimensions),
and the server applies the update using the projected gradient
$P_kP_k^\top g_k \in \mathbb{R}^d$.
Flynn et al.\ analyze the convergence penalty introduced by compression and show that leveraging
matrix-smoothness leads to sharp convergence guarantees in expectation under the
bounded-variance assumption.

\subsection{Stochastic Optimization under Heavy-Tailed Noise}
Classical analyses have established in-expectation convergence for standard SGD 
under the bounded-variance assumption \cite{ghadimi2013}. More recently, 
high-probability convergence has been shown under the sub-Gaussian 
gradient noise assumption \cite{liu2023hp}. However, growing empirical and theoretical evidence indicates that stochastic gradients arising in modern 
large-scale optimization can exhibit heavy-tailed behaviour, which may fall outside the classical bounded-variance assumption. 
This phenomenon has been observed in a variety of settings, including deep neural network training
\cite{simsekli2019,nguyen2019},
motivating a growing body of theoretical work on optimization under heavy-tailed noise.
In particular, clipped SGD and its variants have been extensively analyzed in this context, leading to convergence guarantees both in expectation and with high probability; see, e.g., \cite{zhang2020clipped,nguyen2023}.
However, convergence analyses often rely on carefully tuned clipping thresholds and assume knowledge of problem-dependent parameters (e.g., noise moments or smoothness constants).   

\subsection{Normalized Gradient Methods}
Motivated by these limitations, there has been renewed interest in \emph{normalized} gradient methods, going back to Nesterov's normalized gradient schemes for stochastic optimization \cite{nesterov1984}.
\citet{huebler2025} study NSGD under heavy-tailed noise with bounded $p$-th moments, and establish both in-expectation and high-probability convergence guarantees in two regimes: (i) a fully parameter-free setting where the tail index $p$ is unknown, and (ii) a parameter-aware setting where $p$ is known.
Relatedly, \citet{liu2025clippingfree,sun2024} analyze momentum-based normalized methods under heavy-tailed noise and establish convergence guarantees in expectation; however, corresponding high-probability guarantees remain unavailable.
In fact, for momentum variants, \citet{huebler2025} empirically suggest that obtaining a clean high-probability convergence bound with a simple $\log(1/\delta)$ dependence may be challenging.

\subsection{Compression Meets Heavy Tails: Limitations of Existing Work}
Recently, the interaction between gradient compression and heavy-tailed noise has attracted attention.  
\citet{chen2024safl_v1} analyze clipped SGD with gradient compression under heavy-tailed noise.
However, their high-probability convergence guarantees apply only in the parameter-aware regime where the tail index $p$ is known.
Moreover, their guarantees require choosing the subspace dimension $r$ to satisfy a lower bound
that depends on $d$, $T$, and $\delta$. In other words, for a fixed $r$, the condition can be rearranged
to yield a restriction on the allowable horizon $T$, so the result does not directly provide an
$\epsilon$-dependent iteration complexity guarantee.

As another line of work, 
\citet{kornilov2025sign_v1} analyze sign-based stochastic
methods (SignSGD and its variants) under heavy-tailed noise. They derive high-probability
convergence bounds, but these bounds exhibit an unfavourable dependence on the dimension $d$.
In particular, measured in the Euclidean norm, their bounds include terms of the form $(\sqrt{d}/\epsilon)^{(3p-2)/(p-1)}$ when $p$ is known and $(\sqrt{d}/\epsilon)^{2p/(p-1)}$ when $p$ is unknown; thus, the dependence on $d$ can become significantly worse for small $p$.

\subsection{Randomized Subspaces in Modern Large-Scale Training}
\label{app:llm-subspace}
We cite the following large-scale training works only as context: although they also reduce the update dimension,
their primary focus and algorithmic setting differ from ours.
In a related but distinct line of work, low-dimensional update ideas have recently gained visibility in large-scale
training---especially in LLM optimization---as practical tools for improving memory and/or communication efficiency.
Much of this literature is experiment-driven, but several works also develop stochastic convergence analyses \citep{chen2025greedylore_v4,chen2025rso_llm_v1,wen2025sron,he2025subspace_llm}.
Existing analyses in this literature often focus on in-expectation guarantees under standard assumptions such as bounded variance;
high-probability bounds and heavy-tailed noise are typically not their primary focus. Rather, a central motivation in this line is
system-level efficiency in large-scale training, whereas our motivation is theory-first convergence analysis.

\section{Missing Proofs for RS-SGD}
\label{app:rsgd-hp}

\begin{definition}\label{app:def:norm-subg}
Let $X$ be an $\mathbb{R}^d$-valued random vector. We say that $\|X\|$ is $\sigma$-sub-Gaussian if
\begin{align*}
\forall |\lambda| \le \frac{1}{\sigma}, \qquad
\mathbb{E}\!\left[\exp\!\left(\lambda^2\|X\|^2\right)\right]
\le \exp\!\left(\lambda^2\sigma^2\right).
\end{align*}
\end{definition}

We study how the parameter scales under mini-batching.

\begin{lemma}\label{lem:norm-to-mgf}
Assume $\mathbb{E}[X]=0$ and $\|X\|$ is $\sigma$-sub-Gaussian in the sense of \Cref{app:def:norm-subg}.
Then for all $u\in\mathbb{R}^d$,
\begin{align*}
\mathbb{E}\big[\exp(\langle u,X\rangle)\big]
\le \exp\!\left(\frac{e+1}{2}\,\|u\|^2\sigma^2\right).
\end{align*}
\end{lemma}

\begin{proof}
We borrow ideas from the proof of Proposition~2.6.1 in \cite{vershynin2025hdp}
for the one-dimensional case, adapting them to our setting involving $\|X\|$.

Fix $u\in\mathbb{R}^d$. We start with
\begin{align*}
e^{\langle u,X\rangle}
\overset{(a)}{\le}\;
1+\langle u,X\rangle + \frac{\langle u,X\rangle^2}{2}e^{|\langle u,X\rangle|}
\overset{(b)}{\le}\;
1+\langle u,X\rangle + \frac{\|u\|^2\|X\|^2}{2}e^{\|u\|\|X\|} \refstepcounter{equation}\tag{\theequation}\label{eq:exp-linear}.
\end{align*}
Here (a) is the inequality $e^x \le 1 + x + \frac{x^2}{2}e^{|x|}$ for all $x\in\mathbb{R}$,
and (b) uses $|\langle u,X\rangle|\le \|u\|\,\|X\|$ and $\langle u,X\rangle^2\le \|u\|^2\|X\|^2$.

Next,
\begin{align*}
e^{\|u\|\|X\|}
\overset{(a)}{\le}\;
\exp\!\left(\frac{\sigma^2\|u\|^2}{2}+\frac{\|X\|^2}{2\sigma^2}\right) \refstepcounter{equation}\tag{\theequation}\label{eq:young},
\end{align*}
where (a) follows from $ab \le \frac{\sigma^2 a^2}{2}+\frac{b^2}{2\sigma^2}$ applied to
$a=\|u\|$ and $b=\|X\|$.

From \eqref{eq:exp-linear} and \eqref{eq:young}, we obtain
\begin{align*}
e^{\langle u,X\rangle}
&\le 1+\langle u,X\rangle
+ \frac{\|u\|^2\|X\|^2}{2}\exp\!\left(\frac{\sigma^2\|u\|^2}{2}+\frac{\|X\|^2}{2\sigma^2}\right) \\
&= 1+\langle u,X\rangle
+ \frac{\|u\|^2\sigma^2}{2}e^{\frac{\sigma^2\|u\|^2}{2}}
\cdot \frac{\|X\|^2}{\sigma^2}e^{\frac{\|X\|^2}{2\sigma^2}} \\
&\overset{(a)}{\le}
1+\langle u,X\rangle
+ \frac{\|u\|^2\sigma^2}{2}e^{\frac{\sigma^2\|u\|^2}{2}}
\exp\!\left(\frac{\|X\|^2}{\sigma^2}\right).
\end{align*}
Here (a) uses $\frac{\|X\|^2}{\sigma^2}\le \exp\!\left(\frac{\|X\|^2}{2\sigma^2}\right)$,
which follows from $t^2\le e^{t^2/2}$ applied to $t=\|X\|/\sigma$.

Taking expectations and using $\mathbb{E}[X]=0$ yields
\begin{align*}
\mathbb{E}\big[e^{\langle u,X\rangle}\big]
&\le 1 + \frac{\|u\|^2\sigma^2}{2}e^{\frac{\sigma^2\|u\|^2}{2}}
\mathbb{E}\!\left[\exp\!\left(\frac{\|X\|^2}{\sigma^2}\right)\right].
\end{align*}
By \Cref{app:def:norm-subg} with $\lambda = 1/\sigma$,
$\mathbb{E}[\exp(\|X\|^2/\sigma^2)]\le e$. 
Therefore,
\begin{align*}
\mathbb{E}\big[e^{\langle u,X\rangle}\big]
&\le 1 + \frac{e}{2}\|u\|^2\sigma^2\,e^{\frac{\sigma^2\|u\|^2}{2}}
\overset{(a)}{\le}
e^{\frac{\sigma^2\|u\|^2}{2}}
\left(1+\frac{e}{2}\|u\|^2\sigma^2\right)
\overset{(b)}{\le}
\exp\!\left(\frac{1}{2}\sigma^2\|u\|^2+\frac{e}{2}\sigma^2\|u\|^2\right).
\end{align*}
Here (a) uses that $e^{\frac{\sigma^2\|u\|^2}{2}}\ge 1$, hence
$1+ae^{b}\le e^{b}(1+a)$ with $a=\frac{e}{2}\|u\|^2\sigma^2$, $b=\frac{\sigma^2\|u\|^2}{2}$.
Here (b) uses $1+y\le e^y$ with $y=\frac{e}{2}\|u\|^2\sigma^2$.

\end{proof}

\begin{lemma}\label{lem:mgf-to-norm}
Assume $\mathbb{E}[X]=0$ and there exists $\sigma'>0$ such that
\begin{align*}
\forall u\in\mathbb{R}^d,\qquad
\mathbb{E}\big[\exp(\langle u,X\rangle)\big] \le \exp(\|u\|^2\sigma'^2).
\end{align*}
Then with $\bar{\sigma}:=24\sqrt{e}\,\sigma'$, we have
\begin{align*}
\forall |\lambda|\le \frac{1}{\bar{\sigma}},\qquad
\mathbb{E}\big[\exp(\lambda^2\|X\|^2)\big]
\le 5^d \exp(\lambda^2\bar{\sigma}^2).
\end{align*}
\end{lemma}

\begin{proof}
We follow the one-dimensional argument underlying Proposition~2.6.1 in Vershynin,
and adapt it to our setting. We then combine this with a covering argument
in $\mathbb{R}^d$ to control $\|X\|$.

\paragraph{Step 1}
Fix $u\in\mathbb{R}^d\setminus\{0\}$ and $t>0$. For any $\lambda>0$,
\begin{align}
\mathbb{P}(\langle u,X\rangle \ge t)
&= \mathbb{P}\big(e^{\lambda\langle u,X\rangle}\ge e^{\lambda t}\big)
\overset{(a)}{\le}
e^{-\lambda t}\,\mathbb{E}\!\left[e^{\lambda\langle u,X\rangle}\right]
\overset{(b)}{\le}
\exp\!\left(-\lambda t+\lambda^2\sigma'^2\|u\|^2\right).
\label{eq:step1-markov}
\end{align}
Here (a) is Markov's inequality, and (b) uses the assumed bound
$\mathbb{E}[\exp(\langle v,X\rangle)]\le \exp(\sigma'^2\|v\|^2)$ with $v=\lambda u$.

Optimizing the right-hand side of \eqref{eq:step1-markov} over $\lambda$ gives
\begin{align}
\mathbb{P}(\langle u,X\rangle \ge t)
\le
\exp\!\left(-\frac{t^2}{4\sigma'^2\|u\|^2}\right).
\label{eq:step1-upper}
\end{align}

Applying \eqref{eq:step1-upper} with $u$ and with $-u$ yields, for all $t\ge 0$,
\begin{align}
\mathbb{P}\big(|\langle u,X\rangle|\ge t\big)
\le
2\exp\!\left(-\frac{t^2}{4\sigma'^2\|u\|^2}\right).
\label{eq:proj-tail}
\end{align}

\paragraph{Step 2}
For $p\ge 1$, we use the identity
\begin{align}
\mathbb{E}[|Y|^p]
=
\int_0^\infty \mathbb{P}(|Y|\ge t)\,p t^{p-1}\,dt ,
\label{eq:moment-identity}
\end{align}
and apply it with $Y=\langle u,X\rangle$. Substituting \eqref{eq:proj-tail} into
\eqref{eq:moment-identity} gives
\begin{align}
\mathbb{E}[|\langle u,X\rangle|^p]
&\overset{(a)}{\le}
\int_0^\infty
2\exp\!\left(-\frac{t^2}{4\sigma'^2\|u\|^2}\right)\,p t^{p-1}\,dt
\notag\\
&=
p\Gamma\!\left(\frac{p}{2}\right)(2\sigma'\|u\|)^p.
\label{eq:step2-gamma}
\end{align}
Here (a) uses the tail bound \eqref{eq:proj-tail}.

Using $\Gamma(x)\le 3x^x$ for $x\ge 1/2$, we obtain
\begin{align}
\mathbb{E}[|\langle u,X\rangle|^p]
\le
3p\left(\frac{p}{2}\right)^{p/2}(2\sigma'\|u\|)^p .
\label{eq:step2-momentbound}
\end{align}

Taking $p$-th roots in \eqref{eq:step2-momentbound} and using $(3p)^{1/p}\le 3$ for $p\ge1$ yields
\begin{align}
\forall p\ge1,\qquad
\big(\mathbb{E}[|\langle u,X\rangle|^p]\big)^{1/p}
\le 6\sigma'\|u\|\sqrt{p}.
\label{eq:moment-growth}
\end{align}

\paragraph{Step 3}
Let $\hat{\sigma}:=6\sigma'$. By \eqref{eq:moment-growth},
\begin{align}
\mathbb{E}\big[(\langle u,X\rangle)^{2p}\big]
\overset{(a)}{\le}
\big(\hat{\sigma}\|u\|\sqrt{2p}\big)^{2p}
=
\big(\hat{\sigma}^2\|u\|^2\cdot 2p\big)^p .
\label{eq:step3-evenmom}
\end{align}
Here (a) is \eqref{eq:moment-growth} applied with exponent $2p$.

Therefore,
\begin{align}
\mathbb{E}\big[e^{\lambda^2(\langle u,X\rangle)^2}\big]
&=
1+\sum_{p=1}^\infty
\frac{\lambda^{2p}\,\mathbb{E}[(\langle u,X\rangle)^{2p}]}{p!}
\notag\\
&\overset{(a)}{\le}
1+\sum_{p=1}^\infty
\frac{\lambda^{2p}\,(\hat{\sigma}^2\|u\|^2\cdot 2p)^p}{p!}
\overset{(b)}{\le}
1+\sum_{p=1}^\infty
\big(\lambda^2\cdot 2\hat{\sigma}^2\|u\|^2 e\big)^p .
\label{eq:step3-series}
\end{align}
Here (a) uses \eqref{eq:step3-evenmom}, and (b) uses $p!\ge (p/e)^p$.

Now restrict to $u\in B_d:=\{u:\|u\|\le 1\}$ and define $\sigma:=12\sqrt{e}\,\sigma'$.
Then $2\hat{\sigma}^2 e=\sigma^2/2$, and \eqref{eq:step3-series} gives, for $\|u\|\le 1$,
\begin{align}
\mathbb{E}\big[e^{\lambda^2(\langle u,X\rangle)^2}\big]
\le
1+\sum_{p=1}^\infty \left(\frac{\lambda^2\sigma^2}{2}\right)^p
=
\frac{1}{1-\lambda^2\sigma^2/2}.
\label{eq:step3-geom}
\end{align}

For $|\lambda|\le 1/\sigma$ we have $\lambda^2\sigma^2/2\le 1/2$, hence
\begin{align}
\frac{1}{1-\lambda^2\sigma^2/2}
\overset{(a)}{\le}
\exp(\lambda^2\sigma^2).
\label{eq:step3-final}
\end{align}
Here (a) uses $\frac{1}{1-x}\le e^{2x}$ for $x\in[0,1/2]$ with $x=\lambda^2\sigma^2/2$.

Combining \eqref{eq:step3-geom} and \eqref{eq:step3-final} yields
\begin{align}
\forall u\in B_d,\ \forall |\lambda|\le \frac{1}{\sigma},\qquad
\mathbb{E}\big[e^{\lambda^2(\langle u,X\rangle)^2}\big]\le e^{\lambda^2\sigma^2}.
\label{eq:proj-quad-mgf}
\end{align}

\paragraph{Step 4}
Let $\mathcal{N}_{1/2}$ be a $(1/2)$-net of $B_d$ in Euclidean norm. Then
\begin{align}
|\mathcal{N}_{1/2}|
\overset{(a)}{\le}
\left(1+\frac{2}{1/2}\right)^d
=
5^d .
\label{eq:net-card}
\end{align}
(a)\; standard volumetric bound for Euclidean nets.

Moreover, since $\|X\|=\max_{v\in B_d} \langle v,X\rangle$, the net property implies
\begin{align}
\|X\|
\overset{(a)}{\le}
2\max_{z\in \mathcal{N}_{1/2}} \langle z,X\rangle,
\qquad\text{and hence}\qquad
\|X\|^2
\overset{(b)}{\le}
4\max_{z\in \mathcal{N}_{1/2}} \langle z,X\rangle^2 .
\label{eq:norm-by-net}
\end{align}
(a)\; approximation of the supremum over $B_d$ by a $(1/2)$-net,
(b)\; squaring both sides.

Therefore, for $|\lambda|\le 1/(2\sigma)$,
\begin{align}
\mathbb{E}\big[e^{\lambda^2\|X\|^2}\big]
\overset{(a)}{\le}
\mathbb{E}\Big[e^{4\lambda^2 \max_{z\in\mathcal{N}_{1/2}}\langle z,X\rangle^2}\Big]
\overset{(b)}{\le}
\sum_{z\in\mathcal{N}_{1/2}}
\mathbb{E}\big[e^{4\lambda^2 \langle z,X\rangle^2}\big].
\label{eq:step4-union}
\end{align}
(a)\; using \eqref{eq:norm-by-net},
(b)\; using $e^{\max_i a_i}\le \sum_i e^{a_i}$.

For each $z\in\mathcal{N}_{1/2}\subset B_d$, applying \eqref{eq:proj-quad-mgf}
with parameter $2\lambda$ yields
\begin{align}
\sum_{z\in\mathcal{N}_{1/2}}
\mathbb{E}\big[e^{4\lambda^2 \langle z,X\rangle^2}\big]
\overset{(a)}{\le}
\sum_{z\in\mathcal{N}_{1/2}} e^{4\lambda^2\sigma^2}
=
|\mathcal{N}_{1/2}|\,e^{4\lambda^2\sigma^2}
\overset{(b)}{\le}
5^d e^{4\lambda^2\sigma^2}.
\label{eq:step4-sum}
\end{align}
(a)\; applying \eqref{eq:proj-quad-mgf} with $u=z$,
(b)\; using \eqref{eq:net-card}.

Finally, define $\bar{\sigma}:=2\sigma=24\sqrt{e}\,\sigma'$. Then for $|\lambda|\le 1/\bar{\sigma}$
we have $|\lambda|\le 1/(2\sigma)$ and $4\lambda^2\sigma^2=\lambda^2\bar{\sigma}^2$, so
\begin{align}
\mathbb{E}\big[e^{\lambda^2\|X\|^2}\big]
\le
5^d\,e^{\lambda^2\bar{\sigma}^2}.
\label{eq:step4-final}
\end{align}

\end{proof}

\begin{proposition}[Generalization of \Cref{lem:subgauss_minibatch}]\label{prop:minibatch}
Let $X_1,\dots,X_n$ be independent $\mathbb{R}^d$-valued random vectors such that
$\mathbb{E}[X_i]=0$ and $\|X_i\|$ is $\sigma$-sub-Gaussian for all $i$.
Let $\bar{X} := \frac{1}{n}\sum_{i=1}^n X_i$.
Define
\[
\bar{\sigma}
:= 24\sqrt{\frac{e(e+1)}{2}}\cdot \frac{\sigma}{\sqrt{n}}.
\]
Then for all $|\lambda|\le 1/\bar{\sigma}$,
\begin{align*}
\mathbb{E}\big[e^{\lambda^2\|\bar{X}\|^2}\big]
\le 5^d\, \exp(\lambda^2\bar{\sigma}^2).
\end{align*}
\end{proposition}

\begin{proof}
By \Cref{lem:norm-to-mgf}, for each $i$ and all $u\in\mathbb{R}^d$,
\begin{align}
\mathbb{E}\big[\exp(\langle u,X_i\rangle)\big]
\le \exp\!\left(\frac{e+1}{2}\sigma^2\|u\|^2\right).
\label{eq:Xi-mgf}
\end{align}

Using independence, we obtain
\begin{align}
\mathbb{E}\big[\exp(\langle u,\bar{X}\rangle)\big]
&=
\mathbb{E}\Big[\exp\Big(\frac{1}{n}\sum_{i=1}^n \langle u,X_i\rangle\Big)\Big] \notag\\
&=
\prod_{i=1}^n
\mathbb{E}\Big[\exp\Big(\frac{1}{n}\langle u,X_i\rangle\Big)\Big] \notag\\
&\overset{(a)}{\le}
\prod_{i=1}^n
\exp\!\left(\frac{1}{n^2}\cdot\frac{e+1}{2}\sigma^2\|u\|^2\right)
=
\exp\!\left(\frac{e+1}{2n}\sigma^2\|u\|^2\right).
\label{eq:barX-mgf}
\end{align}
(a)\; applying \eqref{eq:Xi-mgf} with $u/n$.

Thus, the assumption of \Cref{lem:mgf-to-norm} holds for $\bar{X}$ with
\[
\sigma'=\sqrt{\frac{e+1}{2n}}\;\sigma .
\]
Applying \Cref{lem:mgf-to-norm} yields
\[
\bar{\sigma}
=
24\sqrt{e}\,\sigma'
=
24\sqrt{\frac{e(e+1)}{2}}\;\frac{\sigma}{\sqrt{n}},
\]
which completes the proof.

\end{proof}

\begin{lemma}\label{lem:martingale-cher}
Assume $F$ is $L$-smooth and bounded below by $F^\ast$.
Assume $\mathbb{E}[\nabla f(x,\xi)]=\nabla F(x)$ and that
$\|\nabla f(x,\xi)-\nabla F(x)\|$ is $\sigma$-sub-Gaussian (in the sense of \Cref{app:def:norm-subg}).
Let $P_k$ be i.i.d.\ Haar random matrices, and let $\{x_k\}_{k\ge 0}$ be generated by the update rule
\[
x_{k+1}=x_k-\bar{\eta} P_kP_k^\top g_k,
\]
where $g_k$ denotes a mini-batch stochastic gradient estimator at $x_k$.
Define the events
\[
A_k := \Big\{\|P_k^\top \nabla F(x_k)\|^2 \ge \tfrac{1}{2}\|\nabla F(x_k)\|^2\Big\},
\qquad
X:=\sum_{j=0}^{T-1}\mathbf{1}_{A_j}.
\]
Let
\[
\mu \coloneqq
\begin{cases}
1 - I_{\frac{r}{2d}}\!\left(\frac{r}{2}, \frac{d-r}{2}\right), & r<d,\\
1, & r=d,
\end{cases}
\]
where $I_{x}(a,b)$ is the regularized incomplete beta function.
Then for all $s\in(0,1)$,
\[
\mathbb{P}\big(X\le (1-s)\mu T\big)\le \exp\!\left(-\frac{s^2}{2}\mu T\right).
\]
\end{lemma}

\begin{proof}
Let $\mathcal{F}_k := \sigma(g_0,P_0,\dots,g_{k-1},P_{k-1})$.
We first consider the case 
$r<d$.
Define
\[
p_k:=\mathbb{P}(A_k\mid \mathcal{F}_k)
=\mathbb{E}[\mathbf{1}_{A_k}\mid \mathcal{F}_k].
\]
For $\|\nabla F(x_k)\|\neq 0$, the identity $p_k=\mu$ follows from rotational invariance;
this argument appears in the proof of Lemma~1 in \citet{kozak2021}.
If $\|\nabla F(x_k)\|=0$, then $p_k=1$.

Therefore,
\begin{align}
\sum_{k=0}^{T-1} p_k \ge \mu T.
\label{eq:sum-pk}
\end{align}

Let $\psi(\lambda):=e^{-\lambda}-1+\lambda$ and fix $\lambda>0$.
Then
\begin{align}
\mathbb{E}\!\left[e^{-\lambda(\mathbf{1}_{A_k}-p_k)}\mid \mathcal{F}_k\right]
&= p_k e^{-\lambda(1-p_k)} + (1-p_k)e^{\lambda p_k} \notag\\
&= e^{\lambda p_k}\big(p_k e^{-\lambda}+(1-p_k)\big) \notag\\
&\overset{(a)}{\le} e^{\lambda p_k}\exp\!\big(p_k(e^{-\lambda}-1)\big)
= \exp\!\big(\psi(\lambda)p_k\big).
\label{eq:cher-step}
\end{align}
(a)\; using $1+x\le e^x$ with $x=p_k(e^{-\lambda}-1)$.

Hence,
\begin{align}
\mathbb{E}\!\left[e^{-\lambda(\mathbf{1}_{A_k}-p_k)-\psi(\lambda)p_k}\mid \mathcal{F}_k\right]\le 1.
\label{eq:supermgf}
\end{align}
Define
\[
M_k := \exp\!\left(-\lambda\sum_{j=0}^k(\mathbf{1}_{A_j}-p_j)-\psi(\lambda)\sum_{j=0}^k p_j\right).
\]
By \eqref{eq:supermgf}, $\{M_k\}$ is a supermartingale, and thus
\begin{align}
\mathbb{E}[M_{T-1}]\le 1.
\label{eq:MT-1}
\end{align}

Fix $s\in(0,1)$. Using \eqref{eq:sum-pk}, for any $\lambda>0$,
\begin{align}
\mathbb{P}(X \le (1-s)\mu T)
&\le
\mathbb{P}\!\left(X \le (1-s)\sum_{k=0}^{T-1}p_k\right) \notag\\
&=
\mathbb{P}\!\left(e^{-\lambda X} \ge e^{-\lambda(1-s)\sum_{k=0}^{T-1}p_k}\right) \notag\\
&=
\mathbb{P}\!\left(e^{-\lambda\left(X-\sum_{k=0}^{T-1}p_k\right)}
\ge e^{\lambda s\sum_{k=0}^{T-1}p_k}\right) \notag\\
&=
\mathbb{P}\!\left(
e^{-\lambda\left(X-\sum_{k=0}^{T-1}p_k\right)-\psi(\lambda)\sum_{k=0}^{T-1}p_k}
\ge
e^{(\lambda s-\psi(\lambda))\sum_{k=0}^{T-1}p_k}
\right) \notag\\
&\overset{(a)}=
\mathbb{P}\!\left(
M_{T-1} \ge
\exp\!\left((\lambda s-\psi(\lambda))\sum_{k=0}^{T-1}p_k\right)
\right).
\label{eq:reduce-to-M}
\end{align}
Here, (a) follows from the definition of $M_k$.

Choose $\lambda=-\log(1-s)$, so that
\begin{align}
\lambda s-\psi(\lambda)
= s+(1-s)\log(1-s)
\overset{(a)}{\ge} \frac{s^2}{2}.
\label{eq:psi-lower}
\end{align}
(a)\; the standard inequality $s+(1-s)\log(1-s)\ge s^2/2$ for $s\in(0,1)$.

Combining \eqref{eq:reduce-to-M}, \eqref{eq:sum-pk}, and \eqref{eq:psi-lower} yields
\begin{align}
\mathbb{P}(X\le (1-s)\mu T)
&\le \mathbb{P}\!\left(M_{T-1}\ge e^{\frac{s^2}{2}\sum_{k=0}^{T-1}p_k}\right)
\le \mathbb{P}\!\left(M_{T-1}\ge e^{\frac{s^2}{2}\mu T}\right) \notag\\
&\overset{(a)}{\le} \mathbb{E}[M_{T-1}]\, e^{-\frac{s^2}{2}\mu T}
\overset{(b)}{\le} \exp\!\left(-\frac{s^2}{2}\mu T\right).
\label{eq:final-cher}
\end{align}
(a)\; Markov's inequality applied to $M_{T-1}$.
(b)\; using \eqref{eq:MT-1}.
This completes the proof for the case $r<d$.

When $r=d$, we have $\mu=1$ and $P_kP_k^\top=I_d$, hence
$\|P_k^\top \nabla F(x_k)\|=\|\nabla F(x_k)\|$ for all $k$.
Therefore $A_k$ holds for every $k$, and thus $X=T$.
Since $(1-s)\mu T=(1-s)T<T$, it follows that
\[
\mathbb{P}\!\left(X\le (1-s)\mu T\right)=0,
\]
which completes the proof.

\end{proof}

\begin{theorem}[Generalization of \Cref{thm:main_highprob}]\label{app:thm:main_highprob}
Suppose that $F$ satisfies \Cref{ass:lower,ass:Lsmooth}, and that the gradient 
satisfies \Cref{ass:unbiased,ass:subgauss}.
Let $\{x_k\}_{k\ge 0}$ be the sequence generated by the RS-SGD iteration.
Let $\bar{B}\coloneqq \lceil \max\{1, BT^q\} \rceil$ for $B>0$ and $q>0$,
$\alpha \coloneqq 24\sqrt{\frac{e(e+1)}{2}}$,
and $\bar{\eta} = \frac{r}{Ld}$.
Further define
\begin{equation}
\mu \coloneqq
\begin{cases}
1 - I_{\frac{r}{2d}}\!\left(\frac{r}{2}, \frac{d-r}{2}\right), & r<d,\\
1, & r=d,
\end{cases}
\label{app:eq:def_mu}
\end{equation}
where $I_{x}(a,b)$ denotes the regularized incomplete Beta function.
Then for any $\delta, \delta' \in (0,1)$ and $T$ such that
$2\log\!\left(\frac{1}{\delta'}\right) < \mu T$, 
with probability at least $1-\delta-\delta'$, it holds that
\begin{align*}
\min_{0\leq k \leq T-1}\|\nabla F(x_k)\|^2 
\leq 
\frac{
4L\frac{d}{r}\Delta_0
}{
\mu T - \sqrt{2\mu T \log\!\left(\frac{1}{\delta'}\right)}
}
\\
+
\frac{
2\frac{d}{r}\left(\frac{\alpha^2\sigma^2}{\bar{B}}\right)
\bigl(d\log 5 + 1\bigr)
}{
\mu - \sqrt{2\frac{\mu}{T}\log\!\left(\frac{1}{\delta'}\right)}
}
+
\frac{
2\frac{d}{r}\left(\frac{\alpha^2\sigma^2}{\bar{B}}\right)
\log\!\left(\frac{1}{\delta}\right)
}{
\mu T - \sqrt{2\mu T \log\!\left(\frac{1}{\delta'}\right)}
}.
\end{align*}
\end{theorem}

\noindent\emph{Remark.} 
Setting $\delta\leftarrow \delta/2$ and $\delta'\leftarrow \delta/2$ in \Cref{app:thm:main_highprob} yields \Cref{thm:main_highprob}.

\begin{proof}
Let $n_k:=g_k-\nabla F(x_k)$. By $L$-smoothness,
\begin{align}
F(x_{k+1})-F(x_k)
&\le \langle \nabla F(x_k),x_{k+1}-x_k\rangle+\frac{L}{2}\|x_{k+1}-x_k\|^2 \notag\\
&\overset{(a)}{=} -\bar{\eta}\langle \nabla F(x_k),P_kP_k^\top g_k\rangle
+\frac{L\bar{\eta}^2}{2}\|P_kP_k^\top g_k\|^2 \notag\\
&\overset{(b)}{=} -\bar{\eta}\langle \nabla F(x_k),P_kP_k^\top (\nabla F(x_k)+n_k)\rangle
+\frac{L\bar{\eta}^2}{2}\|P_kP_k^\top (\nabla F(x_k)+n_k)\|^2,
\label{eq:rsgd-smooth-descent}
\end{align}
(a)\; uses $x_{k+1}-x_k=-\bar{\eta} P_kP_k^\top g_k$, and
(b)\; uses $g_k=\nabla F(x_k)+n_k$.

Using $P_k^\top P_k=\frac{d}{r}I_r$ and expanding \eqref{eq:rsgd-smooth-descent} yields
\begin{align}
\bar{\eta}\Big(1-\frac{L\bar{\eta}}{2}\cdot\frac{d}{r}\Big)\|P_k^\top\nabla F(x_k)\|^2
+F(x_{k+1})-F(x_k)
\le
\Big(L\bar{\eta}^2\frac{d}{r}-\bar{\eta}\Big)\langle \nabla F(x_k),P_kP_k^\top n_k\rangle
+\frac{L\bar{\eta}^2}{2}\cdot\frac{d}{r}\,\|P_k^\top n_k\|^2.
\label{eq:pre-key-descent}
\end{align}
Substituting $\bar{\eta}=\frac{r}{Ld}$ in \eqref{eq:pre-key-descent} gives
\begin{align}
\frac{\bar{\eta}}{2}\|P_k^\top\nabla F(x_k)\|^2 +F(x_{k+1})-F(x_k)
\le \frac{\bar{\eta}}{2}\|P_k^\top n_k\|^2.
\label{eq:key-descent}
\end{align}

Fix $w>0$ and define
\[
C_k := \frac{\bar{\eta}}{2}\|P_k^\top\nabla F(x_k)\|^2 +F(x_{k+1})-F(x_k),\quad
Z_k := w C_k - d\log 5,\quad
S_k := \sum_{j=k}^{T-1} Z_j.
\]
Let $\mathcal{F}_k:=\sigma(g_0,P_0,\ldots,g_{k-1},P_{k-1})$ and $\bar{\sigma}\coloneqq \alpha\cdot \frac{\sigma}{\sqrt{\bar{B}}}$.
Recall that \Cref{prop:minibatch} provides, for $|\lambda|\le 1/\bar{\sigma}$,
\[
\mathbb{E}\big[\exp(\lambda^2\|n_k\|^2)\mid \mathcal{F}_k\big]
\le 5^d\exp(\lambda^2\bar{\sigma}^2).
\]
In what follows we will apply this bound with
\begin{equation}
\lambda^2 := \frac{w\bar{\eta}}{2}\cdot\frac{d}{r},
\label{eq:lambda-choice}
\end{equation}
which requires the condition
\begin{equation}
\lambda^2 \le \frac{1}{\bar{\sigma}^2}
\quad\Longleftrightarrow\quad
w\le \frac{2}{\bar{\sigma}^2}\frac{r}{\bar{\eta} d}.
\label{eq:w-condition}
\end{equation}

Assume \eqref{eq:w-condition}. Then
\begin{align}
\mathbb{E}\big[e^{Z_k}\mid \mathcal{F}_k\big]
&= \mathbb{E}\big[\exp(wC_k-d\log 5)\mid \mathcal{F}_k\big] \notag\\
&\overset{(a)}{\le}
e^{-d\log 5}\,
\mathbb{E}\Big[\exp\!\big(\tfrac{w\bar{\eta}}{2}\|P_k^\top n_k\|^2\big)\mid \mathcal{F}_k\Big] \notag\\
&\overset{(b)}{\le}
e^{-d\log 5}\,
\mathbb{E}\Big[\exp\!\big(\tfrac{w\bar{\eta}}{2}\cdot\tfrac{d}{r}\|n_k\|^2\big)\mid \mathcal{F}_k\Big] \notag\\
&\overset{(c)}{\le}
e^{-d\log 5}\cdot
\Big(5^d\exp\!\big(\tfrac{w\bar{\eta}}{2}\cdot\tfrac{d}{r}\bar{\sigma}^2\big)\Big)
=
\exp\!\Big(\tfrac{w\bar{\eta}}{2}\cdot\tfrac{d}{r}\bar{\sigma}^2\Big).
\label{eq:expZ}
\end{align}
(a)\; uses \eqref{eq:key-descent},
(b)\; uses $\|P_k^\top n_k\|\le \|P_k^\top\|\|n_k\|=\sqrt{\frac{d}{r}}\|n_k\|$, and
(c)\; applies \Cref{prop:minibatch} with the choice \eqref{eq:lambda-choice}
(which is valid under \eqref{eq:w-condition}).

We next prove by backward induction that, for all $k=0,1,\ldots,T-1$,
\begin{equation}
\mathbb{E}\big[e^{S_k}\mid \mathcal{F}_k\big]
\le
\exp\!\Big(\tfrac{w\bar{\eta}}{2}\cdot\tfrac{d}{r}\bar{\sigma}^2\,(T-k)\Big).
\label{eq:supermgf-ind}
\end{equation}
For the base case $k=T-1$, we have $S_{T-1}=Z_{T-1}$, hence \eqref{eq:supermgf-ind}
follows directly from \eqref{eq:expZ}. Assume \eqref{eq:supermgf-ind} holds for $k+1$.
Then, using $S_k=Z_k+S_{k+1}$ and the tower property,
\begin{align}
\mathbb{E}\big[e^{S_k}\mid \mathcal{F}_k\big]
&=
\mathbb{E}\big[e^{Z_k}e^{S_{k+1}}\mid \mathcal{F}_k\big] \notag\\
&{=}
\mathbb{E}\Big[e^{Z_k}\,\mathbb{E}\big[e^{S_{k+1}}\mid \mathcal{F}_{k+1}\big]\mid \mathcal{F}_k\Big] \notag\\
&\overset{(a)}{\le}
\exp\!\Big(\tfrac{w\bar{\eta}}{2}\cdot\tfrac{d}{r}\bar{\sigma}^2\,(T-k-1)\Big)\,
\mathbb{E}\big[e^{Z_k}\mid \mathcal{F}_k\big] \notag\\
&\overset{(b)}{\le}
\exp\!\Big(\tfrac{w\bar{\eta}}{2}\cdot\tfrac{d}{r}\bar{\sigma}^2\,(T-k)\Big),
\notag
\end{align}
where (a) uses the induction hypothesis, and (b) uses \eqref{eq:expZ}. This completes the induction,
and in particular yields
\begin{equation}
\mathbb{E}[e^{S_0}]
\le
\exp\!\Big(\tfrac{w\bar{\eta}}{2}\cdot\tfrac{d}{r}\bar{\sigma}^2\,T\Big).
\label{eq:supermgf_two}
\end{equation}
Applying Markov's inequality to the nonnegative random variable $\exp(S_0)$, for any $t>0$,
\begin{equation}
\mathbb{P}\big(S_0\ge t\big)
=
\mathbb{P}\big(e^{S_0}\ge e^{t}\big)
\le
e^{-t}\,\mathbb{E}\big[e^{S_0}\big].
\label{eq:markov-detail}
\end{equation}
Using \eqref{eq:supermgf_two} and choosing
$t:=\tfrac{w\bar{\eta}}{2}\cdot\tfrac{d}{r}\bar{\sigma}^2\,T+\log\frac{1}{\delta}$, we obtain
\[
\mathbb{P}\!\left(
S_0\ge \tfrac{w\bar{\eta}}{2}\cdot\tfrac{d}{r}\bar{\sigma}^2\,T+\log\frac{1}{\delta}
\right)
\le \delta.
\]
Therefore, with probability at least $1-\delta$, it holds that
\begin{align}
S_0 \le \tfrac{w\bar{\eta}}{2}\cdot\tfrac{d}{r}\bar{\sigma}^2\,T+\log\frac{1}{\delta}.
\label{eq:S0-markov}
\end{align}
By definition of $S_0$, we have
\begin{align*}
S_0
&=\sum_{k=0}^{T-1}\bigl(wC_k-d\log 5\bigr) \\
&=\sum_{k=0}^{T-1}\Bigl(
w\Bigl(\tfrac{\bar{\eta}}{2}\|P_k^\top\nabla F(x_k)\|^2
+F(x_{k+1})-F(x_k)\Bigr)-d\log 5
\Bigr) \\
&=\frac{w\bar{\eta}}{2}\sum_{k=0}^{T-1}\|P_k^\top\nabla F(x_k)\|^2
+w\bigl(F(x_T)-F(x_0)\bigr)
-dT\log 5.
\end{align*}

Choose
\[
w \coloneqq \frac{2r}{\bar{\sigma}^2 \bar{\eta} d},
\]
so that
\eqref{eq:expZ} is satisfied.
Combining this choice with \eqref{eq:S0-markov}, we conclude that with probability
at least $1-\delta$,
\begin{align*}
\frac{w\bar{\eta}}{2}\sum_{k=0}^{T-1}\|P_k^\top\nabla F(x_k)\|^2
+w\bigl(F(x_T)-F(x_0)\bigr)
-dT\log 5
\le
\frac{w\bar{\eta}}{2}\cdot\frac{d}{r}\bar{\sigma}^2 T
+\log\frac{1}{\delta}.
\end{align*}

Using $F(x_0)-F(x_T)\le F(x_0)-F^\ast\leq\Delta_0$,
$\frac{w\bar{\eta}}{2}=\frac{1}{\bar{\sigma}^2}\cdot\frac{r}{d}$,
and $\bar{\eta}=\frac{r}{Ld}$, this inequality simplifies to
\begin{align}
\sum_{k=0}^{T-1}\|P_k^\top\nabla F(x_k)\|^2
\le
2L\frac{d}{r}\Delta_0
+\frac{d}{r}\bar{\sigma}^2 dT\log 5
+\frac{d}{r}\bar{\sigma}^2 T
+\frac{d}{r}\bar{\sigma}^2\log\frac{1}{\delta}.
\label{eq:sum-proj-grad}
\end{align}

Next, define the events
\[
A_k \coloneqq
\Bigl\{
\|P_k^\top\nabla F(x_k)\|^2
\ge \tfrac{1}{2}\|\nabla F(x_k)\|^2
\Bigr\},
\qquad
X \coloneqq \sum_{k=0}^{T-1}\mathbf{1}_{A_k}.
\]
Then
\begin{align}
\sum_{k=0}^{T-1}\|P_k^\top\nabla F(x_k)\|^2
\ge
\frac{1}{2}\sum_{k=0}^{T-1}\|\nabla F(x_k)\|^2\mathbf{1}_{A_k}
\ge
\frac{X}{2}\min_{0\le k\le T-1}\|\nabla F(x_k)\|^2 .
\label{eq:lower-by-X}
\end{align}

By \Cref{lem:martingale-cher}, for any $s\in(0,1)$,
\begin{equation}
\mathbb{P}\bigl(X\le (1-s)\mu T\bigr)
\le \exp\!\left(-\frac{s^2}{2}\mu T\right).
\label{eq:chernoff-X}
\end{equation}
Fix $\delta'\in(0,1)$ and assume $2\log(1/\delta')<\mu T$.
Setting
\[
s \coloneqq \sqrt{\frac{2}{\mu T}\log\frac{1}{\delta'}},
\]
which satisfies $s\in(0,1)$ under this condition, \eqref{eq:chernoff-X} yields
\begin{equation}
\mathbb{P}\!\left(
X\le \mu T-\sqrt{2\mu T\log\frac{1}{\delta'}}
\right)\le \delta' .
\label{eq:X-lower}
\end{equation}
Therefore, with probability at least $1-\delta'$,
\begin{align}
\sum_{k=0}^{T-1}\|P_k^\top\nabla F(x_k)\|^2
\ge
\frac{1}{2}\Bigl(
\mu T-\sqrt{2\mu T\log\frac{1}{\delta'}}
\Bigr)
\min_{0\le k\le T-1}\|\nabla F(x_k)\|^2 .
\label{eq:proj-lower}
\end{align}

Since $2\log(1/\delta')<\mu T$, we have
\(
\mu T-\sqrt{2\mu T\log(1/\delta')}>0.
\)
Therefore, 
combining \eqref{eq:sum-proj-grad} and \eqref{eq:proj-lower} and applying a union
bound shows that, with probability at least $1-\delta-\delta'$,
\begin{align*}
\min_{0\le k\le T-1}\|\nabla F(x_k)\|^2
\le
\frac{
4L\frac{d}{r}\Delta_0
}{
\mu T-\sqrt{2\mu T\log\frac{1}{\delta'}}
}
+
\frac{
2\frac{d}{r}\bar{\sigma}^2(d\log 5+1)
}{
\mu-\sqrt{2\frac{\mu}{T}\log\frac{1}{\delta'}}
}
+
\frac{
2\frac{d}{r}\bar{\sigma}^2\log\frac{1}{\delta}
}{
\mu T-\sqrt{2\mu T\log\frac{1}{\delta'}}
}.
\end{align*}
This completes the proof.
\end{proof}

\paragraph{Oracle complexity.}
From \Cref{thm:main_highprob} (applied with $q=1$ and $B=\frac{d\sigma^2}{L\Delta_0}$),
the iteration complexity satisfies 
\[
\tilde{\mathcal O}\!\left(
\frac{d}{\mu r}\,\Delta_0 L\,\epsilon^{-2}
\right).
\]
Recalling that the oracle complexity is given by $r\bar{B}T$,
we obtain
\[
\tilde{\mathcal O}\!\left(
\frac{d^{3}}{\mu^{2}r}\,\Delta_0L\sigma^{2}\,\epsilon^{-4}
\right).
\]

\section{Missing Proofs for RS-NSGD}
\label{app:rnsgd}

\begin{proposition}[Restatement of \Cref{prop:tau}]\label{app:prop:tau}
Let \(x \neq 0\). Suppose \(P\) is constructed from a Haar-random orthogonal matrix.
Then \(\|P^\top x\| \neq 0\) almost surely, and there exists a constant
\(\tau > 0\) such that
\[
\mathbb{E}\!\left[\frac{PP^\top x}{\|P^\top x\|}\right]
=
\tau\,\frac{x}{\|x\|}.
\]
Specifically, \(\tau\) is given by
\[
\tau
=
\sqrt{\frac{d}{r}}\,
\frac{
\Gamma\!\left(\frac{r+1}{2} \right)
\Gamma\!\left(\frac{d}{2}\right)
}{
\Gamma\!\left(\frac{r}{2}\right)
\Gamma\!\left(\frac{d+1}{2} \right)
}.
\]
Here, $\Gamma(\cdot)$ denotes the Gamma function.

Moreover,  $\frac{1}{\sqrt{2}} \leq \tau \leq 1.$
\end{proposition}

\begin{proof}
Let $g:\mathbb{R}^d\to\mathbb{R}$ be defined by $g(x)\coloneqq \|P^\top x\|$.
For $x\neq 0$, we have $\|P^\top x\|\neq 0$ almost surely, and thus $g$ is differentiable at $x$ almost surely with
\begin{equation}\label{eq:grad_g}
\nabla g(x)=\frac{PP^\top x}{\|P^\top x\|}.
\end{equation}
Moreover,
\begin{equation}\label{eq:grad_bound}
\left\|\nabla g(x)\right\|
=\left\|\frac{PP^\top x}{\|P^\top x\|}\right\|
\le \|P\|
\le \sqrt{\frac{d}{r}} .
\end{equation}
Hence we may interchange gradient and expectation, obtaining
\begin{equation}\label{eq:swap}
\mathbb{E}\!\left[\frac{PP^\top x}{\|P^\top x\|}\right]
=\mathbb{E}[\nabla g(x)]
=\nabla \mathbb{E}[g(x)].
\end{equation}

We next compute $\mathbb{E}[g(x)]$, which can be written as
\begin{equation}\label{eq:Eg_hom}
\mathbb{E}[g(x)]
=\mathbb{E}[\|P^\top x\|]
=\|x\|\,
\mathbb{E}\!\left[\left\|P^\top \frac{x}{\|x\|}\right\|\right].
\end{equation}
We use the following fact (see the proof of Lemma~1 in \cite{kozak2021}): the random variable
\begin{equation}\label{eq:defY}
Y \coloneqq \frac{r}{d}\left\|P^\top \frac{x}{\|x\|}\right\|^2
\end{equation}
follows a Beta distribution,
\begin{equation}\label{eq:Ybeta}
Y\sim \mathrm{Beta}\!\left(\frac{r}{2},\frac{d-r}{2}\right).
\end{equation}
Let $\alpha\coloneqq \frac{r}{2}$ and $\beta\coloneqq \frac{d-r}{2}$,
we have
\begin{align}
\mathbb{E}[\sqrt{Y}]
&=\int_0^1 \sqrt{y}\,
\frac{y^{\alpha-1}(1-y)^{\beta-1}}{\mathrm{B}(\alpha,\beta)}\,dy
=\int_0^1
\frac{y^{\alpha-\frac12}(1-y)^{\beta-1}}{\mathrm{B}(\alpha,\beta)}\,dy
\notag\\
&=\frac{\mathrm{B}\!\left(\alpha+\frac12,\beta\right)}{\mathrm{B}(\alpha,\beta)}
=\frac{\Gamma\!\left(\alpha+\frac12\right)\Gamma(\alpha+\beta)}
{\Gamma\!\left(\alpha+\beta+\frac12\right)\Gamma(\alpha)}
\notag\\
&=\frac{\Gamma\!\left(\frac{r+1}{2}\right)\Gamma\!\left(\frac{d}{2}\right)}
{\Gamma\!\left(\frac{d+1}{2}\right)\Gamma\!\left(\frac{r}{2}\right)}.
\label{eq:EsqrtY}
\end{align}
By \eqref{eq:defY}, $\left\|P^\top \frac{x}{\|x\|}\right\|=\sqrt{\frac{d}{r}}\,\sqrt{Y}$, hence
\begin{equation}\label{eq:unit_expect}
\mathbb{E}\!\left[\left\|P^\top \frac{x}{\|x\|}\right\|\right]
=
\sqrt{\frac{d}{r}}\,\mathbb{E}[\sqrt{Y}]
=
\sqrt{\frac{d}{r}}\,
\frac{\Gamma\!\left(\frac{r+1}{2}\right)\Gamma\!\left(\frac{d}{2}\right)}
{\Gamma\!\left(\frac{d+1}{2}\right)\Gamma\!\left(\frac{r}{2}\right)}.
\end{equation}
Substituting \eqref{eq:unit_expect} into \eqref{eq:Eg_hom} yields
\begin{equation}\label{eq:Eg_final}
\mathbb{E}[g(x)]
=
\tau\,\|x\|,
\qquad
\tau\coloneqq
\sqrt{\frac{d}{r}}\,
\frac{\Gamma\!\left(\frac{r+1}{2}\right)\Gamma\!\left(\frac{d}{2}\right)}
{\Gamma\!\left(\frac{r}{2}\right)\Gamma\!\left(\frac{d+1}{2}\right)}.
\end{equation}
Taking the gradient of \eqref{eq:Eg_final} and using \eqref{eq:swap}, we conclude that
\[
\mathbb{E}\!\left[\frac{PP^\top x}{\|P^\top x\|}\right]
=\nabla \mathbb{E}[g(x)]
=\tau\,\frac{x}{\|x\|}.
\]

It remains to show $\frac{1}{\sqrt2}\le \tau \le 1$.
For the upper bound, since $\sqrt{\cdot}$ is concave and $Y$ is Beta with mean $\mathbb{E}[Y]=\frac{r}{d}$,
\begin{equation}\label{eq:tau_upper}
\tau
=\sqrt{\frac{d}{r}}\mathbb{E}[\sqrt{Y}]
\overset{(a)}{\le}
\sqrt{\frac{d}{r}}\sqrt{\mathbb{E}[Y]}
=
\sqrt{\frac{d}{r}}\sqrt{\frac{r}{d}}
=1.
\end{equation}
Here, (a) follows from Jensen's inequality.
Therefore, $\tau \le 1$.
To derive the lower bound on $\tau$, we invoke Wendel's inequality \cite{wendel1948}.
By Wendel's inequality, for any $x>0$ and $0<s<1$,
\[
\left(\frac{x}{x+s}\right)^{1-s}\le \frac{\Gamma(x+s)}{x^s\Gamma(x)}\le 1.
\]
Setting $s=\frac12$ and multiplying by $\sqrt{x}$, we obtain
\begin{equation}\label{eq:wendel_s12}
\sqrt{x}\left(\frac{x}{x+\frac12}\right)^{1/2}
\le
\frac{\Gamma\!\left(x+\frac12\right)}{\Gamma(x)}
\le
\sqrt{x}.
\end{equation}
For $x \ge \tfrac12$, the ratio $\frac{x}{x+\frac12}$ is minimized at $x=\tfrac12$,
which gives $\frac{x}{x+\frac12}\ge \frac12$.
Taking square roots on both sides yields
\begin{equation}\label{eq:ratio_lower_aux}
\left(\frac{x}{x+\frac12}\right)^{1/2}\ge \frac{1}{\sqrt2}.
\end{equation}
Combining \eqref{eq:wendel_s12} (left inequality) with \eqref{eq:ratio_lower_aux} and taking $x=\frac{r}{2}\ge \frac12$ gives
\begin{equation}\label{eq:gamma_ratio_r}
\frac{\Gamma\!\left(\frac{r}{2}+\frac12\right)}{\Gamma\!\left(\frac{r}{2}\right)}
\overset{}{\ge}
\frac{1}{\sqrt2}\sqrt{\frac{r}{2}}.
\end{equation}
Moreover, from \eqref{eq:wendel_s12} (right inequality) we obtain
$\Gamma(x)/\Gamma(x+\frac12)\ge 1/\sqrt{x}$ for all $x>0$, hence for $x=\frac{d}{2}$,
\begin{equation}\label{eq:gamma_ratio_d}
\frac{\Gamma\!\left(\frac{d}{2}\right)}{\Gamma\!\left(\frac{d}{2}+\frac12\right)}
\overset{}{\ge}
\sqrt{\frac{2}{d}}.
\end{equation}
Using \eqref{eq:gamma_ratio_r} and \eqref{eq:gamma_ratio_d} in the definition of $\tau$ in \eqref{eq:Eg_final}, we obtain
\begin{align}
\tau
&=
\sqrt{\frac{d}{r}}\,
\frac{\Gamma\!\left(\frac{r}{2}+\frac12\right)}{\Gamma\!\left(\frac{r}{2}\right)}
\cdot
\frac{\Gamma\!\left(\frac{d}{2}\right)}{\Gamma\!\left(\frac{d}{2}+\frac12\right)}
\notag\\
&{\ge}
\sqrt{\frac{d}{r}}
\cdot
\frac{1}{\sqrt2}\sqrt{\frac{r}{2}}
\cdot
\sqrt{\frac{2}{d}}
=\frac{1}{\sqrt2}.
\label{eq:tau_lower}
\end{align}
This completes the proof.
\end{proof}

\begin{proposition}[Restatement of \Cref{prop:secondterm}]
\label{app:prop:secondterm}
Let $S$ be a $d\times d$ square matrix, and $P$ be a Haar matrix. Then,
\[
\mathbb{E}\left[\frac{v^\top PP^\top S PP^\top v}{\|P^\top v\|^2}\right]
=
\frac{d(r-1)}{r(d-1)}\frac{v^\top Sv}{\|v\|^2}
+\frac{d-r}{r(d-1)}\operatorname{tr}(S).
\]
\end{proposition}
\begin{proof}
Fix a vector $v \in \mathbb{R}^d$.  
There exists an orthogonal matrix $\hat Q \in O(d)$ such that
\[
v = \|v\| \hat Q e_1,
\]
where $e_1 \in \mathbb{R}^d$ denotes the vector whose first component is $1$ and whose remaining components are $0$.

By the invariance of the Haar matrix, $\hat Q P$ has the same distribution as $P$.  
Define
\[
\hat S \coloneqq \hat Q^\top S \hat Q .
\]
Then we have
\begin{align}
\mathbb{E}\!\left[\frac{v^\top PP^\top S PP^\top v}{\|P^\top v\|^2}\right]
&= \mathbb{E}\!\left[\frac{v^\top \hat Q PP^\top \hat Q^\top S \hat Q PP^\top \hat Q^\top v}{\|P^\top \hat Q^\top v\|^2}\right] \notag\\
&= \mathbb{E}\!\left[\frac{\|v\| e_1^\top PP^\top \hat Q^\top S \hat Q PP^\top \|v\| e_1}{\|P^\top \|v\| e_1\|^2}\right] \notag\\
&= \mathbb{E}\!\left[\frac{e_1^\top PP^\top \hat S PP^\top e_1}{\|P^\top e_1\|^2}\right].
\label{eq:LHS}
\end{align}

Define the random vector
\[
u \coloneqq \frac{PP^\top e_1}{\|P^\top e_1\|}.
\]
Then
\begin{align}
\mathbb{E}\!\left[\frac{e_1^\top PP^\top \hat S PP^\top e_1}{\|P^\top e_1\|^2}\right]
&= \mathbb{E}[u^\top \hat S u] \notag\\
&= \operatorname{tr}\!\left(\mathbb{E}[uu^\top]\hat S\right),
\label{eq:RHS}
\end{align}
where we used the identity
$\mathbb{E}[u^\top \hat S u]
= \mathbb{E}[\operatorname{tr}(u^\top \hat S u)]
= \mathbb{E}[\operatorname{tr}(uu^\top \hat S)]
= \operatorname{tr}(\mathbb{E}[uu^\top]\hat S)$.

Combining \eqref{eq:LHS} and \eqref{eq:RHS}, we obtain
\[
\mathbb{E}\!\left[\frac{v^\top PP^\top S PP^\top v}{\|P^\top v\|^2}\right]
= \operatorname{tr}(M\hat S),
\quad
M \coloneqq \mathbb{E}[uu^\top].
\]

We now characterize the structure of $M$.
Consider the subset
\[
\mathcal{S} \coloneqq \{Q \in O(d) \mid Q e_1 = e_1\}.
\]
For any $Q \in \mathcal{S}$, by the invariance of the Haar matrix,
\begin{align*}
Q^\top M Q
&= \mathbb{E}[Q^\top uu^\top Q] \\
&= \mathbb{E}\!\left[\frac{Q^\top PP^\top e_1 e_1^\top PP^\top Q}{\|P^\top e_1\|^2}\right] \\
&= \mathbb{E}\!\left[\frac{Q^\top PP^\top Q e_1 e_1^\top Q^\top PP^\top Q}{\|P^\top Q e_1\|^2}\right] \\
&= \mathbb{E}\!\left[\frac{(Q^\top P)(Q^\top P)^\top e_1 e_1^\top (Q^\top P)(Q^\top P)^\top}{\|(Q^\top P)^\top e_1\|^2}\right] \\
&= \mathbb{E}\!\left[\frac{P P^\top e_1 e_1^\top PP^\top}{\|P^\top e_1\|^2}\right] \\
&= \mathbb{E}[uu^\top] \\
&= M.
\end{align*}

This invariance implies strong structural constraints on $M$.
First, fix $i \in \{2,\dots,d\}$ and define the matrix
\[
A_i \coloneqq \diag(1,\dots,1,-1,1,\dots,1) \in O(d-1),
\]
where the $(i-1)$-th diagonal entry is $-1$ and the remaining $d-2$ diagonal
entries are equal to $1$.
Define
\[
Q \coloneqq
\begin{pmatrix}
1 & \mathbf{0} \\
\mathbf{0} & A_i
\end{pmatrix}.
\]
Then $Q \in \mathcal{S}$ by construction.

Since $Q$ has $-1$ only at the $i$-th diagonal entry and $1$ at all other
diagonal entries, we compute $Q^\top M Q$ explicitly.
Writing $M=(m_{jk})_{1\leq j,k\leq d}$, we have
\begin{align*}
M
= Q^\top M Q
= \begin{pmatrix}
  m_{1,1} & \cdots & m_{1,i-1} & -m_{1,i} & m_{1,i+1} & \cdots & m_{1,d} \\[4pt]
  \vdots  &        & \vdots    & \vdots   & \vdots    &        & \vdots \\[4pt]
  m_{i-1,1} & \cdots & m_{i-1,i-1} & -m_{i-1,i} & m_{i-1,i+1} & \cdots & m_{i-1,d} \\[6pt]
  -m_{i,1} & \cdots & -m_{i,i-1} & m_{i,i} & -m_{i,i+1} & \cdots & -m_{i,d} \\[6pt]
  m_{i+1,1} & \cdots & m_{i+1,i-1} & -m_{i+1,i} & m_{i+1,i+1} & \cdots & m_{i+1,d} \\[4pt]
  \vdots &        & \vdots & \vdots & \vdots &        & \vdots \\[4pt]
  m_{d,1} & \cdots & m_{d,i-1} & -m_{d,i} & m_{d,i+1} & \cdots & m_{d,d}
\end{pmatrix}.
\end{align*}
Since this equality holds for all $i \in \{2,\dots,d\}$, it follows that
\[
m_{jk}=0 \quad \text{for all } j\neq k,
\]
and hence $M$ must be a diagonal matrix.
Therefore, $M$ can be written as
\[
M=
\begin{pmatrix}
m_{1,1} & 0 & \cdots & 0 \\
0 & m_{2,2} & \cdots & 0 \\
\vdots & \vdots & \ddots & \vdots \\
0 & 0 & \cdots & m_{d,d}
\end{pmatrix}.
\]

Next, take $i,j \in \{2,\dots,d\}$ with $i \le j$, and let $Q$ be the permutation
matrix corresponding to the transposition $(i,j)$, i.e.,
\[
Q\coloneqq
\begin{pmatrix}
1 &        &        &        &        &        &        \\
  & \ddots &        &        &        &        &        \\
  &        & 0      &        & 1      &        &        \\
  &        &        & \ddots &        &        &        \\
  &        & 1      &        & 0      &        &        \\
  &        &        &        &        & \ddots &        \\
  &        &        &        &        &        & 1
\end{pmatrix}.
\]
The lower-right block of $Q$ is a $(d-1)\times(d-1)$ permutation matrix, and
hence belongs to $O(d-1)$.
Therefore, $Q \in \mathcal{S}$.

Using again the invariance $Q^\top M Q=M$, we obtain
\begin{align*}
M
&= Q^\top M Q \\
&= \diag(
m_{1,1},\dots,m_{i-1,i-1},
m_{j,j},
m_{i+1,i+1},\dots,m_{j-1,j-1},
m_{i,i},
m_{j+1,j+1},\dots,m_{d,d}
).
\end{align*}
Since this holds for arbitrary $i,j \in \{2,\dots,d\}$ with $i \le j$, we conclude
that
\[
m_{2,2}=m_{3,3}=\cdots=m_{d,d}.
\]

Defining $\alpha \coloneqq m_{1,1}$ and $\beta \coloneqq m_{2,2}$, we finally
obtain
\[
M=
\begin{pmatrix}
\alpha & 0 & \cdots & 0 \\
0 & \beta & \cdots & 0 \\
\vdots & \vdots & \ddots & \vdots \\
0 & 0 & \cdots & \beta
\end{pmatrix}
= (\alpha-\beta)e_1e_1^\top + \beta I_d.
\]

We now determine $\alpha$ and $\beta$.
First,
\begin{align*}
e_1^\top M e_1
&= \mathbb{E}[e_1^\top uu^\top e_1] \\
&= \mathbb{E}\!\left[e_1^\top \frac{PP^\top e_1 e_1^\top PP^\top}{\|P^\top e_1\|^2} e_1\right] \\
&= \mathbb{E}[\|P^\top e_1\|^2] \\
&= e_1^\top \mathbb{E}[PP^\top] e_1 \\
&= 1.
\end{align*}
On the other hand,
\[
e_1^\top M e_1 = \alpha,
\]
and hence $\alpha=1$.

Next, we compute the trace of $M$:
\begin{align*}
\operatorname{tr}(M)
&= \operatorname{tr}(\mathbb{E}[uu^\top]) \\
&= \mathbb{E}[\operatorname{tr}(uu^\top)] \\
&= \mathbb{E}[\|u\|^2] \\
&= \mathbb{E}\!\left[\frac{e_1^\top PP^\top PP^\top e_1}{\|P^\top e_1\|^2}\right] \\
&= \frac{d}{r}.
\end{align*}
On the other hand,
\begin{align*}
\operatorname{tr}(M)
&= \operatorname{tr}((1-\beta)e_1e_1^\top + \beta I_d) \\
&= (1-\beta) + \beta d.
\end{align*}
Equating the two expressions yields
\[
(d-1)\beta = \frac{d}{r} - 1,
\quad
\beta = \frac{d-r}{r(d-1)}.
\]

Therefore,
\[
M
= \frac{d(r-1)}{r(d-1)} e_1e_1^\top
+ \frac{d-r}{r(d-1)} I_d.
\]

Finally,
\begin{align*}
\mathbb{E}\!\left[\frac{v^\top PP^\top S PP^\top v}{\|P^\top v\|^2}\right]
&= \operatorname{tr}(M\hat S) \\
&= \operatorname{tr}\!\left(
\frac{d(r-1)}{r(d-1)} e_1e_1^\top \hat S
+ \frac{d-r}{r(d-1)} \hat S
\right) \\
&= \frac{d(r-1)}{r(d-1)} e_1^\top \hat S e_1
+ \frac{d-r}{r(d-1)} \operatorname{tr}(\hat S) \\
&= \frac{d(r-1)}{r(d-1)} \frac{v^\top S v}{\|v\|^2}
+ \frac{d-r}{r(d-1)} \operatorname{tr}(S),
\end{align*}
which completes the proof.
\end{proof}

\begin{theorem}[Restatement of \Cref{Lsmoothrnsgd}]
\label{app:Lsmoothrnsgd}
Suppose that $F$ satisfies \Cref{ass:lower,ass:Lsmooth} 
, and that the gradient  
satisfies \Cref{ass:unbiased,ass:pBCM}.
Let $\{x_k\}_{k\ge 0}$ be the sequence generated by the RS-NSGD iteration.
We set
\[
\bar{B} = \left\lceil \max\{1,\, B T^q\} \right\rceil,
\qquad
\bar{\eta} = \eta T^{-u},
\]
for parameters $B > 0$, $q > 0$, $\eta > 0$, and $u \in (0,1)$.
Then it holds that
\begin{align*}
\frac{1}{T} \sum_{k=0}^{T-1} \mathbb{E}[\|\nabla F(x_k)\|]
\le
\frac{\Delta_0}{\tau \eta T^{1-u}}
\\ + \frac{\eta \ell(\mathbb{L})\|\mathbb{L}\|}{2\tau T^u}
+ \frac{4\sigma}{\max\{1,BT^q\}^{(p-1)/p}}.
\end{align*}\end{theorem}
\begin{proof}
Define
\[
G_k \coloneqq
\begin{cases}
0, & \|P_k^\top g_k\|=0,\\[2mm]
\dfrac{P_kP_k^\top g_k}{\|P_k^\top g_k\|}, & \|P_k^\top g_k\|\neq 0,
\end{cases}
\qquad
x_{k+1}=x_k-\bar{\eta}G_k .
\]
We use the almost sure implications
\[
\|g_k\|=0 \ \Rightarrow\ \|P_k^\top g_k\|=0 \quad \text{a.s.},
\qquad
\|g_k\|\neq 0 \ \Rightarrow\ \|P_k^\top g_k\|\neq 0 \quad \text{a.s.}
\]
so that on $\{\|g_k\|\neq 0\}$ we have $G_k=\frac{P_kP_k^\top g_k}{\|P_k^\top g_k\|}$ a.s.,
and on $\{\|g_k\|=0\}$ we have $G_k=0$ a.s.

Since $F$ satisfies \Cref{ass:Lsmooth}, we have
\begin{align*}
F(x_{k+1})
&\le
F(x_k)
+ \nabla F(x_k)^\top (x_{k+1}-x_k)
+ \frac{1}{2}(x_{k+1}-x_k)^\top \mathbb{L}(x_{k+1}-x_k) \\
&=
F(x_k)
- \bar{\eta}\,\nabla F(x_k)^\top G_k
+ \frac{\bar{\eta}^2}{2}\,G_k^\top \mathbb{L} G_k .
\end{align*}

Rearranging both sides and summing over $k=0,\dots,T-1$, we obtain
\begin{align}
\bar{\eta}
\sum_{k=0}^{T-1}
\nabla F(x_k)^\top G_k
&\le
\sum_{k=0}^{T-1}\bigl(F(x_k)-F(x_{k+1})\bigr)
+ \sum_{k=0}^{T-1}
\frac{\bar{\eta}^2}{2}\,G_k^\top \mathbb{L} G_k
\notag\\
&=
F(x_0)-F(x_T)
+ \sum_{k=0}^{T-1}
\frac{\bar{\eta}^2}{2}\,G_k^\top \mathbb{L} G_k
\notag\\
&\le
F(x_0)-F_*
+ \sum_{k=0}^{T-1}
\frac{\bar{\eta}^2}{2}\,G_k^\top \mathbb{L} G_k
\notag\\
&\le
\Delta_0
+ \sum_{k=0}^{T-1}
\frac{\bar{\eta}^2}{2}\,G_k^\top \mathbb{L} G_k .
\label{eq:main_sum}
\end{align}

Let
\[
\mathcal{F}_k \coloneqq \sigma(g_0,P_0,\dots,g_{k-1},P_{k-1}),
\qquad
\mathcal{F}_k^+ \coloneqq \sigma(g_0,P_0,\dots,g_{k-1},P_{k-1},g_k),
\]
so that $\mathcal{F}_k \subset \mathcal{F}_k^+$.
By the tower property,
\begin{align}
\mathbb{E}\!\left[
G_k^\top \mathbb{L} G_k
\right]
&=
\mathbb{E}\!\left[
\mathbb{E}\!\left[
G_k^\top \mathbb{L} G_k
\;\middle|\;
\mathcal{F}_k^+
\right]
\right]
\notag\\
&\overset{(a)}{=}
\mathbb{E}\!\left[
\mathbf{1}_{\{\|g_k\|\neq 0\}}
\left(
\frac{d(r-1)}{r(d-1)}\frac{g_k^\top \mathbb{L} g_k}{\|g_k\|^2}
+\frac{d-r}{r(d-1)}\operatorname{tr}(\mathbb{L})
\right)
\right]
\notag\\
&\overset{(b)}{\leq}
\ell(\mathbb{L})\|\mathbb{L}\|.
\label{eq:Lmat_bnd}
\end{align}
Here, (a) follows from \Cref{app:prop:secondterm} on the event $\{\|g_k\|\neq0\}$
(using $\|g_k\|\neq 0 \Rightarrow \|P_k^\top g_k\|\neq 0$ a.s., hence
$G_k=\frac{P_kP_k^\top g_k}{\|P_k^\top g_k\|}$ a.s.), and from $G_k=0$ a.s.\ on
$\{\|g_k\|=0\}$.
For (b), note that on $\{\|g_k\|\neq0\}$,
\[
\frac{g_k^\top \mathbb{L} g_k}{\|g_k\|^2}\le \|\mathbb{L}\|,
\]
and the term is multiplied by $\mathbf{1}_{\{\|g_k\|\neq 0\}}\le 1$; together with
the definition of $\ell(\mathbb{L})$ this yields \eqref{eq:Lmat_bnd}.

To justify the application of the tower property, we verify integrability.
We have
\[
\bigl|\nabla F(x_k)^\top G_k\bigr|
\le
\|\nabla F(x_k)\|\,\|G_k\|.
\]
Moreover, if $\|P_k^\top g_k\|=0$ then $G_k=0$ and $\|G_k\|=0$.
If $\|P_k^\top g_k\|\neq0$, then
\[
\|G_k\|
=
\frac{\|P_kP_k^\top g_k\|}{\|P_k^\top g_k\|}
\le
\|P_k\|
\le
\sqrt{\frac{d}{r}},
\]
and hence in all cases,
\[
\bigl|\nabla F(x_k)^\top G_k\bigr|
\le
\sqrt{\frac{d}{r}}\|\nabla F(x_k)\|.
\]
Moreover, by smoothness,
\begin{align*}
\|\nabla F(x_{k+1})\|
&\le
\|\nabla F(x_k)\|
+ L\|x_{k+1}-x_k\| \\
&=
\|\nabla F(x_k)\|
+ \bar{\eta}L\|G_k\|
\le
\|\nabla F(x_k)\|
+ \bar{\eta}L\sqrt{\frac{d}{r}},
\end{align*}
and hence by induction,
\[
\|\nabla F(x_k)\|
\le
\|\nabla F(x_0)\| + k\bar{\eta}L\sqrt{\frac{d}{r}}.
\]
This implies integrability.

We may therefore apply the tower property to obtain
\begin{align}
\mathbb{E}\!\left[\nabla F(x_k)^\top G_k \mid \mathcal{F}_k\right]
=
\mathbb{E}\!\left[
\mathbb{E}\!\left[\nabla F(x_k)^\top G_k \mid \mathcal{F}_k^+\right]
\;\middle|\;
\mathcal{F}_k
\right].
\label{eq:tower_Gk}
\end{align}

We claim that,
\begin{align}
\mathbb{E}\!\left[\nabla F(x_k)^\top G_k \mid \mathcal{F}_k^+\right]
\ge
\tau\Bigl(\|\nabla F(x_k)\|-2\|g_k-\nabla F(x_k)\|\Bigr).
\label{eq:inner_descent_lb}
\end{align}
To prove this, we consider two cases.

\smallskip
\noindent\textbf{Case 1: $\|g_k\|=0$.}
Then $\|P_k^\top g_k\|=0$ almost surely, hence $G_k=0$ and thus
\[
\mathbb{E}\!\left[\nabla F(x_k)^\top G_k \mid \mathcal{F}_k^+\right]=0.
\]
On the other hand, since $g_k=0$,
\[
\|\nabla F(x_k)\|-2\|g_k-\nabla F(x_k)\|
=
\|\nabla F(x_k)\|-2\|\nabla F(x_k)\|
=
-\|\nabla F(x_k)\|.
\]
Therefore the right-hand side of \eqref{eq:inner_descent_lb} equals
$-\tau\|\nabla F(x_k)\|\le 0$, and the inequality holds.

\smallskip
\noindent\textbf{Case 2: $\|g_k\|\neq 0$.}
Then $\|P_k^\top g_k\|\neq 0$ almost surely, so
$G_k=\frac{P_kP_k^\top g_k}{\|P_k^\top g_k\|}$ almost surely, and by \Cref{app:prop:tau},
\[
\mathbb{E}\!\left[G_k\mid \mathcal{F}_k^+\right]
=
\tau\,\frac{g_k}{\|g_k\|}.
\]
Hence,
\begin{align*}
\mathbb{E}\!\left[\nabla F(x_k)^\top G_k \mid \mathcal{F}_k^+\right]
&=
\nabla F(x_k)^\top
\mathbb{E}\!\left[G_k\mid \mathcal{F}_k^+\right]
=
\tau\,\frac{\nabla F(x_k)^\top g_k}{\|g_k\|}.
\end{align*}
Moreover, for any $a,b\in\mathbb{R}^d$ with $b\neq 0$, the elementary inequality (see, e.g., \citet{huebler2025})
\[
\frac{a^\top b}{\|b\|}
\ge
\|a\|-2\|a-b\|
\]
holds. Applying it with $a=\nabla F(x_k)$ and $b=g_k$ yields
\[
\frac{\nabla F(x_k)^\top g_k}{\|g_k\|}
\ge
\|\nabla F(x_k)\|-2\|g_k-\nabla F(x_k)\|.
\]
Multiplying by $\tau$ proves \eqref{eq:inner_descent_lb} in this case.

\smallskip
\noindent
Therefore, combining \eqref{eq:tower_Gk} and \eqref{eq:inner_descent_lb}, we obtain
\begin{align}
\mathbb{E}\!\left[\nabla F(x_k)^\top G_k \mid \mathcal{F}_k\right]
&=
\mathbb{E}\!\left[
\mathbb{E}\!\left[\nabla F(x_k)^\top G_k \mid \mathcal{F}_k^+\right]
\;\middle|\;
\mathcal{F}_k
\right] \notag\\
&\ge
\tau\Bigl(
\|\nabla F(x_k)\|
-
2\,\mathbb{E}\!\left[\|g_k-\nabla F(x_k)\|\mid\mathcal{F}_k\right]
\Bigr).
\label{eq:descent_conditional_Fk}
\end{align}

Following the same argument as in \citet{huebler2025}, \Cref{ass:pBCM}
implies
\[
\mathbb{E}\!\left[\|g_k-\nabla F(x_k)\| \mid \mathcal{F}_k\right]
\le
\frac{2\sigma}{\bar{B}^{(p-1)/p}}.
\]
Substituting this bound into \eqref{eq:descent_conditional_Fk}, we obtain
\begin{align}
\mathbb{E}\!\left[\nabla F(x_k)^\top G_k \mid \mathcal{F}_k\right]
\ge
\tau\left(
\|\nabla F(x_k)\|
-
\frac{4\sigma}{\bar{B}^{(p-1)/p}}
\right).
\label{eq:descent_conditional_final}
\end{align}
Taking expectations of both sides and using the tower property, we obtain
\begin{align}
\mathbb{E}\!\left[\nabla F(x_k)^\top G_k\right]
\ge
\tau\left(
\mathbb{E}\!\left[\|\nabla F(x_k)\|\right]
-
\frac{4\sigma}{\bar{B}^{(p-1)/p}}
\right).
\label{eq:descent_unconditional}
\end{align}
Combining \eqref{eq:main_sum} and \eqref{eq:Lmat_bnd}, taking expectations,
and using \eqref{eq:descent_unconditional}, we obtain
\begin{align*}
\bar{\eta}\tau
\sum_{k=0}^{T-1}
\left(
\mathbb{E}[\|\nabla F(x_k)\|]
-
\frac{4\sigma}{\bar{B}^{(p-1)/p}}
\right)
\le
\Delta_0
+
\frac{\bar{\eta}^2\ell(\mathbb{L})\|\mathbb{L}\|T}{2}.
\end{align*}

Dividing both sides by $\bar{\eta}\tau T$, we obtain
\begin{align}
\frac{1}{T}
\sum_{k=0}^{T-1}
\mathbb{E}[\|\nabla F(x_k)\|]
\le
\frac{\Delta_0}{\bar{\eta}\tau T}
+ \frac{\bar{\eta}\,\ell(\mathbb{L})\|\mathbb{L}\|}{2\tau}
+ \frac{4\sigma}{\bar{B}^{(p-1)/p}}.
\label{eq:pre_subst}
\end{align}

Recall that $\bar{\eta}=\eta T^{-u}$ and
$\bar{B}=\lceil \max\{1,BT^q\}\rceil$.
Moreover, since
\[
\bar{B}
\ge
\max\{1,BT^q\},
\]
it follows that
\[
\bar{B}^{(p-1)/p}
\ge
\max\{1,BT^q\}^{(p-1)/p}.
\]
Substituting these relations into \eqref{eq:pre_subst}, we obtain
\[
\frac{1}{T}
\sum_{k=0}^{T-1}
\mathbb{E}[\|\nabla F(x_k)\|]
\le
\frac{\Delta_0}{\tau\eta T^{1-u}}
+ \frac{\eta\,\ell(\mathbb{L})\|\mathbb{L}\|}{2\tau T^u}
+ \frac{4\sigma}{\max\{1,BT^q\}^{(p-1)/p}}.
\]
\end{proof}

\paragraph{Oracle complexity (parameters, including $p$, are unknown).}
From \Cref{Lsmoothrnsgd}, taking $(u,q)=\bigl(\tfrac12,1\bigr)$ and treating $\eta$ and $B$ as constants,
the iteration complexity satisfies 
\[
\mathcal{O}\!\left(
\left(\frac{\Delta_0}{\tau \epsilon}\right)^{2}
+
\left(\frac{\ell(\mathbb{L})\,\|\mathbb{L}\|}{\tau \epsilon}\right)^{2}
+
\left(\frac{\sigma}{\epsilon}\right)^{\frac{p}{p-1}}
\right).
\]
Recalling that the oracle complexity is given by $r\bar{B}T$,
we obtain
\[
\mathcal{O}\!\left(
r\left(\frac{\Delta_0}{\tau \epsilon}\right)^{4}
+
r\left(\frac{\ell(\mathbb{L})\,\|\mathbb{L}\|}{\tau \epsilon}\right)^{4}
+
r\left(\frac{\sigma}{\epsilon}\right)^{\frac{2p}{p-1}}
\right).
\]

\paragraph{Oracle complexity (parameters, including $p$, are known).}
From \Cref{Lsmoothrnsgd}, we set $u=\tfrac{1}{2}$ and choose
\[
q=\frac{p}{2(p-1)},\qquad
\eta=\sqrt{\frac{\Delta_0}{\ell(\mathbb{L})\,\|\mathbb{L}\|}},\qquad
B=\left(
\frac{\sigma\,\tau}{\sqrt{\Delta_0\,\ell(\mathbb{L})\,\|\mathbb{L}\|}}
\right)^{\frac{p}{p-1}}.
\]
With these choices, the iteration complexity satisfies 
\[
\mathcal{O}\!\left(\left(
\frac{\sqrt{\ell(\mathbb{L})\,\Delta_0\,\|\mathbb{L}\|}}{\tau\epsilon}
\right)^{2}\right).
\]
Recalling that the oracle complexity is given by $r\bar{B}T$,
we obtain
\[
\mathcal{O}\!\left(
r\,\frac{\ell(\mathbb{L})}{\tau^2}\,
\frac{\Delta_0\,\|\mathbb{L}\|}{\epsilon^{2}}\,
\left(\frac{\sigma}{\epsilon}\right)^{\frac{p}{p-1}}
\right).
\]

\begin{lemma}\label{lem:init-grad-bound}
Suppose that $F$ satisfies \Cref{ass:lower,ass:Lsmooth}.
Then
\[
\|\nabla F(x_0)\|
\;\le\;
\sqrt{2 \Delta_0 \, L}.
\]
\end{lemma}

\begin{proof}
By $L$-smoothness, for any $x$ and $v$,
\[
F(x+v)
\le
F(x) + \langle \nabla F(x), v \rangle
+ \frac{L}{2} \|v\|^2 .
\]
Choosing $x=x_0$ and $v = - L^{-1} \nabla F(x_0)$ yields
\[
F_*
\le
F\!\left(x_0 - \frac{1}{L}\nabla F(x_0)\right)
\le
F(x_0) - \frac{1}{2L}\|\nabla F(x_0)\|^2 .
\]
Rearranging proves the claim.
\end{proof}

\begin{theorem}[Restatement of \Cref{high:minibatch-RNSGD-prob}]\label{app:high:minibatch-RNSGD-prob}
Suppose that $F$ satisfies \Cref{ass:lower,ass:Lsmooth}, and that the gradient
satisfies \Cref{ass:unbiased,ass:pBCM}.
Let $\{x_k\}_{k\ge 0}$ be the sequence generated by the RS-NSGD iteration.
We set
\[
\bar{B} = \left\lceil \max\{1,\, B T^q\} \right\rceil,
\qquad
\bar{\eta} = \eta T^{-u},
\]
for parameters $B > 0$, $q > 0$, $\eta > 0$, and $u \in (0,1)$.
Then, with probability at least $1 - \delta$, the following holds:
\begin{align*}
\frac{1}{T} \sum_{k=0}^{T-1} \|\nabla F(x_k)\|
\le \frac{2 \Delta_0}{\tau \eta T^{1-u}}
+ \frac{\frac{d}{r} \eta L}{\tau T^u} \big( 1 + 12\frac{\sqrt{\frac{d}{r}}}{\tau} \log (1/\delta) \big)
\\+ 17\frac{\frac{d}{r} \sqrt{\Delta_0 L}}{\tau^2 T} \log (1/\delta)
+ \frac{8 \sigma}{\max \{1, B T^q\}^{\frac{p-1}{p}}}.
\end{align*}
\end{theorem}

\begin{proof}
Define the filtrations
\[
\mathcal{F}_k \coloneqq \sigma(g_0,P_0,\ldots,g_{k-1},P_{k-1}),
\qquad
\mathcal{F}_k^+ \coloneqq \sigma(g_0,P_0,\ldots,g_{k-1},P_{k-1},g_k).
\]

Define the normalized projected direction
\[
G_k \coloneqq
\begin{cases}
0, & \|P_k^\top g_k\|=0,\\[2mm]
\dfrac{P_kP_k^\top g_k}{\|P_k^\top g_k\|}, & \|P_k^\top g_k\|\neq 0,
\end{cases}
\]
and the RS-NSGD update
\[
x_{k+1} = x_k - \bar{\eta} G_k.
\]
Define
\[
\psi_k \coloneqq \nabla F(x_k)^\top G_k,
\qquad
\phi_k \coloneqq \mathbb{E}[\psi_k\mid\mathcal{F}_k].
\]

We use the following almost sure implications:
\[
\|g_k\|=0 \ \Rightarrow\ \|P_k^\top g_k\|=0 \quad \text{a.s.},
\qquad
\|g_k\|\neq 0 \ \Rightarrow\ \|P_k^\top g_k\|\neq 0 \quad \text{a.s.}
\]

By $L$-smoothness,
\[
F(x_{k+1})
\le
F(x_k)
-\bar{\eta}\nabla F(x_k)^\top G_k
+\frac{L}{2}\bar{\eta}^2\|G_k\|^2.
\]
Summing this inequality for $k=0,\ldots,T-1$ and using $F(x_T)\ge F_*$ yields
\begin{align}
\bar{\eta}\sum_{k=0}^{T-1}\psi_k
\le
\Delta_0
+
\frac{L}{2}\bar{\eta}^2\sum_{k=0}^{T-1}\|G_k\|^2.
\label{eq:descent-basic}
\end{align}

If $\|P_k^\top g_k\|\neq 0$, then
\[
\|G_k\|
=
\frac{\|P_kP_k^\top g_k\|}{\|P_k^\top g_k\|}
\le
\|P_k\|
\le
\sqrt{\frac{d}{r}}.
\]
If $\|P_k^\top g_k\|=0$, then $G_k=0 \leq \sqrt{\frac{d}{r}}$.
Therefore $\|G_k\|^2\le \frac{d}{r}$ for all $k$, and
\begin{align}
\bar{\eta}\sum_{k=0}^{T-1}\psi_k
\le
\Delta_0
+
\frac{\bar{\eta}^2L}{2}\frac{d}{r}T.
\label{eq:descent-NSGD}
\end{align}

Define $D_0\coloneqq 0$ and, for $k\ge1$,
\[
D_k \coloneqq -\bar{\eta}(\psi_{k-1}-\phi_{k-1}).
\]
Then $(D_k,\mathcal{F}_k)$ is a martingale difference sequence.
Let
\[
S_k \coloneqq 2\sqrt{\frac{d}{r}}\,\bar{\eta}\,\|\nabla F(x_{k-1})\|,
\qquad k\ge1.
\]
Then $S_k$ is $\mathcal{F}_{k-1}$-measurable. Moreover,
\[
|\psi_{k-1}|
\le
\|\nabla F(x_{k-1})\|\,\|G_{k-1}\|
\le
\sqrt{\frac{d}{r}}\|\nabla F(x_{k-1})\|,
\]
which also implies
$|\phi_{k-1}|\le \sqrt{\frac{d}{r}}\|\nabla F(x_{k-1})\|$.
Hence
\[
\exp\!\left(\frac{D_k^2}{S_k^2}\right)
=
\exp\!\left(
\frac{(\psi_{k-1}-\phi_{k-1})^2}
{4\frac{d}{r}\|\nabla F(x_{k-1})\|^2}
\right)
\le
\exp\!\left(
\frac{(|\psi_{k-1}|+|\phi_{k-1}|)^2}
{4\frac{d}{r}\|\nabla F(x_{k-1})\|^2}
\right)
\le
e.
\]

By Lemma~11 of \cite{huebler2025}\footnote{
To avoid division by zero and to make the argument rigorous,
one may regularize the denominator by replacing
$S_k$ with $S_k + \tfrac{1}{n}$ for $n\in\mathbb{N}$.
This yields a monotone decreasing sequence of events indexed by $n$,
to which Lemma~11 applies for each fixed $n$.
Letting $n\to\infty$ and using the continuity of probability
yields the stated inequality.
}, for any $\lambda>0$ and any $\delta\in(0,1)$,
with probability at least $1-\delta$,
\begin{align}
\sum_{k=1}^T D_k
\le
\frac{3}{4}\lambda\sum_{k=1}^T S_k^2
+
\frac{1}{\lambda}\log\frac{1}{\delta}.
\label{eq:hp-mds}
\end{align}

Combining \eqref{eq:hp-mds} with \eqref{eq:descent-NSGD} yields
\begin{align}
\bar{\eta}\sum_{k=0}^{T-1}
\Bigl(
\phi_k
-
3\frac{d}{r}\lambda\bar{\eta}\|\nabla F(x_k)\|^2
\Bigr)
\le
\Delta_0
+
\frac{\bar{\eta}^2L}{2}\frac{d}{r}T
+
\frac{1}{\lambda}\log\frac{1}{\delta}.
\label{eq:combined}
\end{align}

Let $C_T\coloneqq \bar{\eta}^2T$ and choose
\[
\lambda
\coloneqq
\frac{\tau}{
6\bigl(
\bar{\eta}\frac{d}{r}\|\nabla F(x_0)\|
+
C_TL\bigl(\tfrac{d}{r}\bigr)^{3/2}
\bigr)
}.
\]
Using the gradient growth bound
\[
\|\nabla F(x_k)\|
\le
\|\nabla F(x_0)\|
+
k\bar{\eta}L\sqrt{\frac{d}{r}},
\]
we obtain
\begin{align*}
3\frac{d}{r}\lambda\bar{\eta}\|\nabla F(x_k)\|^2
&=
\frac{\tau 3\frac{d}{r}\bar{\eta}\|\nabla F(x_k)\|^2}
{6\bigl(\bar{\eta}\frac{d}{r}\|\nabla F(x_0)\|
+ C_TL(\frac{d}{r})^{3/2}\bigr)} \\
&\overset{(b)}{\le}
\frac{\tau \frac{d}{r}\bar{\eta}
(\|\nabla F(x_0)\|+k\bar{\eta}L\sqrt{\frac{d}{r}})}
{2\bigl(\bar{\eta}\frac{d}{r}\|\nabla F(x_0)\|
+ C_TL(\frac{d}{r})^{3/2}\bigr)}
\|\nabla F(x_k)\| \\
&\overset{(c)}{\le}
\frac{\tau}{2}\|\nabla F(x_k)\|.
\end{align*}
Here, $(b)$ uses
\[
\|\nabla F(x_k)\|^2
\le
\bigl(\|\nabla F(x_0)\|
+
k\bar{\eta}L\sqrt{\tfrac{d}{r}}\bigr)
\|\nabla F(x_k)\|,
\]
and $(c)$ follows from
\[
\frac{d}{r}\bar{\eta}k\bar{\eta}L\sqrt{\tfrac{d}{r}}
\le
C_TL\bigl(\tfrac{d}{r}\bigr)^{3/2}.
\]

As shown in \eqref{eq:descent_conditional_final} of \Cref{app:Lsmoothrnsgd},
\[
\phi_k
\ge
\tau\!\left(
\|\nabla F(x_k)\|
-
\frac{4\sigma}{\bar{B}^{\frac{p-1}{p}}}
\right).
\]
Substituting these bounds into \eqref{eq:combined} and multiplying both sides by
$\frac{2}{\tau\bar{\eta}T}$ yield
\begin{align*}
\frac{1}{T}\sum_{k=0}^{T-1}\|\nabla F(x_k)\|
\le
\frac{2\Delta_0}{\tau\bar{\eta}T}
+
\frac{\bar{\eta}L}{\tau}\frac{d}{r}
+
\frac{12\frac{d}{r}\|\nabla F(x_0)\|}{\tau^2T}\log\frac{1}{\delta}
+
\frac{12\bar{\eta}L}{\tau^2}
\Bigl(\frac{d}{r}\Bigr)^{3/2}
\log\frac{1}{\delta}
+
\frac{8\sigma}{\bar{B}^{\frac{p-1}{p}}}.
\end{align*}

Finally, using \Cref{lem:init-grad-bound},
$\|\nabla F(x_0)\|\le \sqrt{2\Delta_0L}$,
and $12\sqrt{2}\le17$, together with the definitions of
$\bar{\eta}$ and $\bar{B}$, the stated bound follows.
\end{proof}

\paragraph{Oracle complexity (parameters, including $p$, are unknown).}
From \Cref{high:minibatch-RNSGD-prob}, taking $(u,q)=\bigl(\tfrac12,1\bigr)$ and treating $\eta$ and $B$ as constants,
the iteration complexity satisfies 
\[
\tilde{\mathcal{O}}\!\left(
\left(\frac{\Delta_0}{\tau \epsilon}\right)^{2}
+
\left(\frac{\left(\frac{d}{r}\right)^{\frac{3}{2}}L}{\tau^2 \epsilon}\right)^{2}
+
\left(\frac{\sigma}{\epsilon}\right)^{\frac{p}{p-1}}
\right).
\]
Recalling that the oracle complexity is given by $r\bar{B}T$,
we obtain
\[
\tilde{\mathcal{O}}\!\left(
  r\left(\frac{\Delta_0}{\tau\epsilon}\right)^4
  +
  \frac{d^6}{r^5}\left(\frac{L}{\tau^2 \epsilon}\right)^4
  +
  r\left(\frac{\sigma}{\epsilon}\right)^{\frac{2p}{p-1}}
\right).
\]

\paragraph{Oracle complexity (parameters, including $p$, are known).}
From \Cref{high:minibatch-RNSGD-prob}, we set $u=\tfrac{1}{2}$ and choose
\[
q=\frac{p}{2(p-1)},\qquad
\eta=\left(\frac{r}{d}\right)^{\frac{3}{4}}
\sqrt{\frac{\Delta_0 \tau}{L}},\qquad
B=\left(
\frac{\sigma \tau^{\frac{3}{2}}}{\sqrt{\Delta_0 L}}
\left(\frac{r}{d}\right)^{\frac{3}{4}}
\right)^{\frac{p}{p-1}}.
\]
With these choices, the iteration complexity satisfies 
\[
\tilde{\mathcal{O}}\!\left(
  \left(\frac{\left(\frac{d}{r}\right)^{\frac{3}{4}}\sqrt{\Delta_0L}}{\tau^{\frac{3}{2}} \epsilon}\right)^{2}
\right).
\]

Recalling that the oracle complexity is given by $r\bar{B}T$,
we obtain
\[
\tilde{\mathcal{O}}\!\left(
  \left(\frac{d}{r^{1/3}\tau^2}\right)^{3/2}
  \frac{\Delta_0 L}{\epsilon^2}
  \left(\frac{\sigma}{\epsilon}\right)^{\frac{p}{p-1}}
\right).
\]

\section{RS-NGD: Update Rule and Convergence Analysis}
\label{app:rngd}
\paragraph{Randomized Subspace Normalized Gradient Descent (RS-NGD).}
At each iteration $k$, we sample $P_k \in \mathbb{R}^{d\times r}$ as an i.i.d. Haar matrix,
and perform the update
\begin{align*}
x_{k+1} &= x_k - \bar{\eta}\, G_k,\\
G_k &\coloneqq
\begin{cases}
\dfrac{P_kP_k^\top \nabla F(x_k)}{\|P_k^\top \nabla F(x_k)\|},
& \text{if } \|P_k^\top \nabla F(x_k)\|\neq 0,\\[2mm]
0,
& \text{if } \|P_k^\top \nabla F(x_k)\|=0.
\end{cases}
\end{align*}

\begin{theorem}\label{thm:RNGD_exp}
Assume that $F$ satisfies \Cref{ass:lower,ass:Lsmooth}.
Let $\{x_k\}_{k\ge0}$ be generated by RS-NGD from $x_0$.
We set
\[
\bar{\eta} = \eta T^{-u},
\]
for parameters $\eta > 0$, and $u \in (0,1)$.
Then, for all $T\ge1$,
\[
 \frac{1}{T}\sum_{k=0}^{T-1} \mathbb{E}\!\left[\|\nabla F(x_k)\|\right]
 \le
 \frac{\Delta_0 }{\eta \tau T^{1-u}}
 +\frac{\eta \ell(\mathbb{L})\|\mathbb{L}\|}{2\tau T^u}.
\]
\end{theorem}
\begin{proof}
Let $\mathcal{F}_k \coloneqq \sigma(P_0,\dots,P_{k-1})$.
Note that $x_k$ is $\mathcal{F}_k$-measurable.

By \Cref{ass:Lsmooth}, for any $x,v$,
\[
F(x+v)
\le
F(x)+\langle \nabla F(x),v\rangle+\frac{1}{2}v^\top\mathbb{L}v .
\]
Applying this with $x=x_k$ and $v=x_{k+1}-x_k=-\bar{\eta}\,G_k$, we obtain
\begin{align*}
F(x_{k+1})
&\le
F(x_k)-\bar{\eta}\,\langle \nabla F(x_k),G_k\rangle
+\frac{\bar{\eta}^2}{2}\,G_k^\top\mathbb{L}G_k.
\end{align*}
By the definition of $G_k$, on $\{\|P_k^\top\nabla F(x_k)\|\neq0\}$ we have
$G_k=\frac{P_kP_k^\top\nabla F(x_k)}{\|P_k^\top\nabla F(x_k)\|}$, hence
\[
\langle \nabla F(x_k),G_k\rangle
=
\frac{\nabla F(x_k)^\top P_kP_k^\top \nabla F(x_k)}{\|P_k^\top \nabla F(x_k)\|}
=
\|P_k^\top \nabla F(x_k)\|.
\]
On $\{\|P_k^\top\nabla F(x_k)\|=0\}$ we have $G_k=0$, and thus both sides equal $0$.
Therefore, for all outcomes,
\[
\langle \nabla F(x_k),G_k\rangle=\|P_k^\top\nabla F(x_k)\|.
\]
Substituting this identity yields
\[
F(x_{k+1})
\le
F(x_k)-\bar{\eta}\,\|P_k^\top\nabla F(x_k)\|
+\frac{\bar{\eta}^2}{2}\,G_k^\top\mathbb{L}G_k.
\]
Rearranging gives
\[
\bar{\eta}\,\|P_k^\top\nabla F(x_k)\|
\le
F(x_k)-F(x_{k+1})
+\frac{\bar{\eta}^2}{2}\,G_k^\top\mathbb{L}G_k.
\]
Summing over $k=0,\dots,T-1$ yields
\begin{align}
\bar{\eta}\sum_{k=0}^{T-1}\|P_k^\top\nabla F(x_k)\|
&\le
\Delta_0
+\frac{\bar{\eta}^2}{2}\sum_{k=0}^{T-1}G_k^\top\mathbb{L}G_k .
\label{eq:sum_rngd}
\end{align}

Next, the Haar property in \Cref{app:prop:tau}, in particular
\eqref{eq:Eg_final}, implies that
\[
\mathbb{E}\!\left[\|P_k^\top\nabla F(x_k)\|\mid \mathcal{F}_k\right]
=
\tau\,\|\nabla F(x_k)\|.
\]
Taking expectations and using the tower property yields
\[
\mathbb{E}\!\left[\|P_k^\top\nabla F(x_k)\|\right]
=
\tau\,\mathbb{E}\!\left[\|\nabla F(x_k)\|\right].
\]

Moreover, by \Cref{app:prop:secondterm} and the tower property,
\[
\mathbb{E}\!\left[G_k^\top\mathbb{L}G_k\right]
=
\mathbb{E}\!\left[
\mathbb{E}\!\left[G_k^\top\mathbb{L}G_k \mid \mathcal{F}_k\right]
\right]
\le
\ell(\mathbb{L})\,\|\mathbb{L}\|.
\]

Taking expectations in \eqref{eq:sum_rngd} and combining the above bounds,
we obtain
\[
\bar{\eta}\tau \sum_{k=0}^{T-1}\mathbb{E}\!\left[\|\nabla F(x_k)\|\right]
\le
\Delta_0
+\frac{\bar{\eta}^2}{2}\,\ell(\mathbb{L})\,\|\mathbb{L}\|\,T.
\]
Dividing both sides by $\bar{\eta}\tau T$ completes the proof.
\end{proof}

We next establish a high-probability convergence bound for RS-NGD.

\begin{theorem}
\label{thm:RNGD_hp}
Assume that $F$ satisfies \Cref{ass:lower,ass:Lsmooth}.
Let $\{x_k\}_{k\ge0}$ be generated by RS-NGD from $x_0$.
Fix $T\ge1$ and set
\[
\bar{\eta}=\eta T^{-u},
\qquad
\text{for parameters }\eta>0,\ u\in(0,1).
\]
Then, for any $\delta\in(0,1)$, with probability at least $1-\delta$,
\begin{align*}
\frac{1}{T}\sum_{k=0}^{T-1}\|\nabla F(x_k)\|
\le\;
\frac{2\Delta_0}{\eta\tau T^{1-u}}
+\frac{\eta}{\tau T^{u}}\frac{d}{r}L
+\frac{3}{\tau^2}\left(
\frac{d}{r}\frac{\sqrt{2\Delta_0L}}{T}
+\frac{\eta}{T^{u}}\Bigl(\frac{d}{r}\Bigr)^{3/2}L
\right)\log\frac{1}{\delta}.
\end{align*}
\end{theorem}

\begin{proof}
Let $\mathcal{F}_k\coloneqq\sigma(P_0,\dots,P_{k-1})$.

By \Cref{ass:Lsmooth}, for any $x,v$,
\[
F(x+v)\le F(x)+\langle \nabla F(x),v\rangle+\frac12 L\|v\|^2 .
\]
Applying this with $x=x_k$ and $v=x_{k+1}-x_k=-\bar{\eta}G_k$, we obtain
\begin{align*}
F(x_{k+1})
&\le
F(x_k)-\bar{\eta}\,\langle \nabla F(x_k),G_k\rangle
+\frac{\bar{\eta}^2}{2}\,L\|G_k\|^2 .
\end{align*}

By the same argument as in the proof of \Cref{thm:RNGD_exp}, we have
\[
\langle \nabla F(x_k), G_k \rangle
=
\|P_k^\top \nabla F(x_k)\|.
\]
Moreover, since $\|P_k\|\le\sqrt{d/r}$, it follows that
$\|G_k\|\le\sqrt{d/r}$ and hence
\[
L\|G_k\|^2 \le L\frac{d}{r}.
\]
Substituting these bounds yields
\[
F(x_{k+1})
\le
F(x_k)-\bar{\eta}\,\|P_k^\top\nabla F(x_k)\|
+\frac{\bar{\eta}^2}{2}\frac{d}{r}L.
\]
Rearranging gives
\[
\bar{\eta}\,\|P_k^\top\nabla F(x_k)\|
\le
F(x_k)-F(x_{k+1})
+\frac{\bar{\eta}^2}{2}\frac{d}{r}L.
\]
Summing over $k=0,\dots,T-1$ and using $F(x_T)\ge F^*$ yields
\begin{equation}\label{eq:hp_rngd_base}
\bar{\eta}\sum_{k=0}^{T-1}\|P_k^\top\nabla F(x_k)\|
\le
\Delta_0+\frac{\bar{\eta}^2}{2}\frac{d}{r}LT.
\end{equation}

Define
\[
\psi_k \coloneqq \|P_k^\top\nabla F(x_k)\|,
\qquad
\phi_k \coloneqq \mathbb{E}[\psi_k\mid \mathcal{F}_k],
\qquad
D_k \coloneqq \bar{\eta}(\phi_{k-1}-\psi_{k-1}),\quad k\ge1,
\]
with $D_0\coloneqq0$.

Then $(D_k,\mathcal{F}_k)$ is a martingale difference sequence.
Moreover, since $0\le\psi_k\le\sqrt{d/r}\|\nabla F(x_k)\|$, we have
\[
|D_k|
\le
\bar{\eta}\sqrt{\frac{d}{r}}\,\|\nabla F(x_{k-1})\|.
\]

Define
\[
S_k \coloneqq \bar{\eta}\sqrt{\frac{d}{r}}\,\|\nabla F(x_{k-1})\|.
\]
Then $S_k$ is $\mathcal{F}_{k-1}$-measurable and satisfies
\[
\frac{D_k^2}{S_k^2}
\le
1.
\]
Consequently,
\[
\mathbb{E}\!\left[
\exp\!\left(\frac{D_k^2}{S_k^2}\right)
\Bigm|\mathcal{F}_{k-1}
\right]
\le e .
\]

Applying Lemma~11 in \emph{From Gradient Clipping to Normalization},
for any $\lambda>0$, with probability at least $1-\delta$,
\begin{equation}\label{eq:hp_mds}
\sum_{t=1}^T D_t
\le
\frac{3}{4}\lambda\sum_{t=1}^T S_t^2
+
\frac{1}{\lambda}\log\frac{1}{\delta}.
\end{equation}

\footnote{
Similar to \Cref{app:high:minibatch-RNSGD-prob},
to avoid division by zero and to make the argument rigorous,
one may regularize the denominator by replacing
$S_k$ with $S_k + \tfrac{1}{n}$ for $n\in\mathbb{N}$.
This yields a monotone decreasing sequence of events indexed by $n$,
to which Lemma~11 applies for each fixed $n$.
Letting $n\to\infty$ and using the continuity of probability
yields the stated inequality.
}

As in the proof of \Cref{thm:RNGD_exp},
\[
\phi_k
=
\mathbb{E}[\psi_k \mid \mathcal{F}_k]
=
\tau\,\|\nabla F(x_k)\|.
\]
Combining \eqref{eq:hp_mds} with \eqref{eq:hp_rngd_base},
we obtain, with probability at least $1-\delta$,
\begin{align}\label{eq:RNGDconv}
\bar{\eta}\tau\sum_{k=0}^{T-1}\|\nabla F(x_k)\|
&\le
\Delta_0
+
\frac{\bar{\eta}^2}{2}\frac{d}{r}LT
+
\frac{3}{4}\lambda\bar{\eta}^2\frac{d}{r}
\sum_{k=0}^{T-1}\|\nabla F(x_k)\|^2
+
\frac{1}{\lambda}\log\frac{1}{\delta}.
\end{align}

By $L$-smoothness, 
\[
\|\nabla F(x_{k+1})\|
\le
\|\nabla F(x_k)\|+\bar{\eta} L\sqrt{\frac{d}{r}}.
\]
Iterating this bound yields, for all $k\in\{0,\dots,T\}$,
\begin{equation}
\|\nabla F(x_k)\|
\le
\|\nabla F(x_0)\|
+
k\,\bar{\eta}  L\sqrt{\frac{d}{r}}
\le
\|\nabla F(x_0)\|
+
T\,\bar{\eta}  L\sqrt{\frac{d}{r}}.
\label{eq:grad_growth}
\end{equation}

We now choose
\[
\lambda
\coloneqq
\frac{2}{3}\,
\frac{\tau}{
\bar{\eta} \,\frac{d}{r}\Bigl(\|\nabla F(x_0)\|+\bar{\eta}  T L\sqrt{\frac{d}{r}}\Bigr)
}.
\]
Then, using \eqref{eq:grad_growth}, for each $k=0,\dots,T-1$ we have
\begin{align*}
\frac{3}{4}\lambda\,\bar{\eta} \,\frac{d}{r}\,\|\nabla F(x_k)\|
&\le
\frac{3}{4}\lambda\,\bar{\eta} \,\frac{d}{r}\Bigl(\|\nabla F(x_0)\|+\bar{\eta}  T L\sqrt{\frac{d}{r}}\Bigr) \\
&=
\frac{3}{4}\cdot
\frac{2}{3}\,\tau
=
\frac{\tau}{2}.
\end{align*}
Substituting this pointwise bound into \eqref{eq:RNGDconv} gives
\[
\bar{\eta} \frac{\tau}{2}\sum_{k=0}^{T-1}\|\nabla F(x_k)\|
\le
\Delta_0
+\frac{\bar{\eta} ^2TL}{2}\frac{d}{r}
+\frac{1}{\lambda}\log\frac{1}{\delta}.
\]

Multiplying both sides by $\frac{2}{\bar{\eta}\tau T}$, we obtain
\begin{align*}
\frac{1}{T}\sum_{k=0}^{T-1}\|\nabla F(x_k)\|
&\le
\frac{2\Delta_0}{\bar{\eta}\tau T}
+
\frac{\bar{\eta}}{\tau}\frac{d}{r}L
+
\frac{2}{\bar{\eta}\tau T}\cdot\frac{1}{\lambda}\log\frac{1}{\delta}.
\end{align*}

We have
\[
\frac{1}{\lambda}
=
\frac{3}{2}\,
\frac{\bar{\eta}\frac{d}{r}\bigl(\|\nabla F(x_0)\|+\bar{\eta}TL\sqrt{\frac{d}{r}}\bigr)}{\tau}.
\]
Substituting this into the last term yields
\begin{align*}
\frac{2}{\bar{\eta}\tau T}\cdot\frac{1}{\lambda}\log\frac{1}{\delta}
&=
\frac{3}{\tau^2}\left(\frac{d}{r}\frac{\|\nabla F(x_0)\|}{T}
+
\bar{\eta}L\Bigl(\frac{d}{r}\Bigr)^{3/2}\right)\log\frac{1}{\delta}.
\end{align*}

Therefore,
\begin{align*}
\frac{1}{T}\sum_{k=0}^{T-1}\|\nabla F(x_k)\|
&\le
\frac{2\Delta_0}{\bar{\eta}\tau T}
+
\frac{\bar{\eta}}{\tau}\frac{d}{r}L
+
\frac{3}{\tau^2}\left(\frac{d}{r}\frac{\|\nabla F(x_0)\|}{T}
+
\bar{\eta}L\Bigl(\frac{d}{r}\Bigr)^{3/2}\right)\log\frac{1}{\delta}.
\end{align*}

Substituting $\bar{\eta}=\eta T^{-u}$ gives
\begin{align*}
\frac{1}{T}\sum_{k=0}^{T-1}\|\nabla F(x_k)\|
\le\;
\frac{2\Delta_0}{\eta\tau T^{1-u}}
+\frac{\eta}{\tau T^{u}}\frac{d}{r}L
+\frac{3}{\tau^2}\left(
\frac{d}{r}\frac{\|\nabla F(x_0)\|}{T}
+\frac{\eta}{T^{u}}\Bigl(\frac{d}{r}\Bigr)^{3/2}L
\right)\log\frac{1}{\delta}.
\end{align*}

Finally, by \Cref{lem:init-grad-bound}, we have
\[
\|\nabla F(x_0)\|
\le
\sqrt{2\Delta_0L}.
\]
Substituting this bound into the last display yields
\begin{align*}
\frac{1}{T}\sum_{k=0}^{T-1}\|\nabla F(x_k)\|
\le\;
\frac{2\Delta_0}{\eta\tau T^{1-u}}
+\frac{\eta}{\tau T^{u}}\frac{d}{r}L
+\frac{3}{\tau^2}\left(
\frac{d}{r}\frac{\sqrt{2\Delta_0L}}{T}
+\frac{\eta}{T^{u}}\Bigl(\frac{d}{r}\Bigr)^{3/2}L
\right)\log\frac{1}{\delta}.
\end{align*}
This completes the proof.
\end{proof}

\section{Experimental Details}
\label{app:exp_details}

\subsection{Initialization.}
For each random seed, all methods start from the same initialization in both the synthetic quadratic experiment and character-level language modeling.

\subsection{Stepsize tuning}
For each method (and each subspace dimension $r$ when applicable), we tune the stepsize $\bar{\eta}$ over the
prescribed grid using three independent tuning runs, and then evaluate the selected $\bar{\eta}$ on five disjoint
evaluation runs under the same oracle-call budget.

For the synthetic quadratic experiment, the tuning score of a candidate step size $\bar{\eta}$ is defined as the average of
$\|\nabla F(x)\|=\|\Lambda x\|$ over the last $10$ iterations of the run.
We then average this score over the tuning runs (across tuning seeds) and select the $\bar{\eta}$ that minimizes the resulting averaged score. The resulting best stepsizes are summarized in \Cref{tab:bestlr-quad-horiz}.The selected stepsize is then fixed and used for all five evaluation runs.

For character-level language modeling, the tuning score of a candidate $\bar{\eta}$ is defined as the final
\emph{estimated validation loss} at the end of the oracle-call budget.
The validation loss is estimated by averaging the per-character negative log-likelihood over five randomly sampled minibatches.
We average this score over tuning seeds and select the $\bar{\eta}$ with the smallest averaged score. The resulting best stepsizes are summarized in \Cref{tab:bestlr-charlm-horiz}.
The selected stepsize is fixed and used for all five evaluation runs.

\begin{table}[t]
\centering
\caption{Best stepsizes for the synthetic quadratic with heavy-tailed noise ($d=100$, $\bar{B}=4$), selected by validation tuning.}
\small
\begin{tabular}{l c c c c}
\toprule
Method & $r$ & Best $\bar{\eta}$ ($\rho=4$) & Best $\bar{\eta}$ ($\rho=20$) & Best $\bar{\eta}$ ($\rho=100$) \\
\midrule
SGD      & $-$  & $10^{-3}$ & $10^{-2}$ & $10^{-3}$ \\
RS-SGD   & $20$ & $10^{-3}$ & $10^{-3}$ & $10^{-4}$ \\
RS-SGD   & $4$  & $10^{-5}$ & $10^{-4}$ & $10^{-5}$ \\
\midrule
NSGD     & $-$  & $10^{-2}$ & $10^{-2}$ & $10^{-2}$ \\
RS-NSGD  & $20$ & $10^{-2}$ & $10^{-3}$ & $10^{-3}$ \\
RS-NSGD  & $4$  & $10^{-3}$ & $10^{-3}$ & $10^{-4}$ \\
\bottomrule
\end{tabular}

\begin{minipage}{0.98\linewidth}
\centering
\textbf{Tuning grid: }
$\bar{\eta} \in \{10^{-6},10^{-5},10^{-4},10^{-3},10^{-2},10^{-1}\}$.
\end{minipage}

\label{tab:bestlr-quad-horiz}
\end{table}
\begin{table}[t]
\centering
\caption{Best stepsizes for CharLM ($d=5079$ for PTB and $d=13700$ for WikiText-2, $\bar{B}=128$), selected by validation tuning (minimizing the final validation loss).}
\small
\begin{tabular}{l c c c}
\toprule
Method & $r$ & Best $\bar{\eta}$ (PTB) & Best $\bar{\eta}$ (WikiText-2) \\
\midrule
SGD      & $-$   & $1$        & $1$ \\
RS-SGD   & $500$ & $1$        & $1$ \\
RS-SGD   & $50$  & $10^{-1}$  & $10^{-1}$ \\
RS-SGD   & $5$   & $10^{-2}$  & $10^{-2}$ \\
\midrule
NSGD     & $-$   & $1$        & $1$ \\
RS-NSGD  & $500$ & $1$        & $1$ \\
RS-NSGD  & $50$  & $10^{-1}$  & $10^{-1}$ \\
RS-NSGD  & $5$   & $10^{-2}$  & $10^{-2}$ \\
\bottomrule
\end{tabular}

\begin{minipage}{0.98\linewidth}
\centering
\textbf{Tuning grid: }
$\bar{\eta} \in \{10^{-3},10^{-2},10^{-1},1,10\}$.
\end{minipage}

\label{tab:bestlr-charlm-horiz}
\end{table}

\subsection{Logging and evaluation}
For the synthetic quadratic experiment, we log the true gradient norm
$\|\nabla F(x)\|=\|\Lambda x\|$ along the trajectory. In all plots, each curve reports the mean over five
evaluation runs, and the shaded region indicates $\pm 1$ standard deviation.

For character-level language modeling, we log training/validation/test losses at each logging point using the same evaluation procedure applied separately to each data split.
To reduce evaluation overhead on CPU, we compute each reported loss as the average per-character negative log-likelihood over five randomly sampled mini-batches.
In all plots, each curve reports the mean over five evaluation runs (five random seeds), and the shaded region indicates $\pm 1$ standard deviation across seeds.
\subsection{Validation/test results}
At each logging point, we record the validation and test losses in addition to the training loss using the
same evaluation procedure described above. \Cref{fig:charlm-valid-test} reports the resulting validation and test trajectories.
Consistent with the training results, RS-NSGD with $r\in\{5,50\}$ achieves the lowest losses under the oracle-call budget
across both datasets.

\begin{figure*}[t]
  \centering

  \begin{minipage}[t]{0.4\textwidth}
    \centering
    \includegraphics[width=\linewidth]{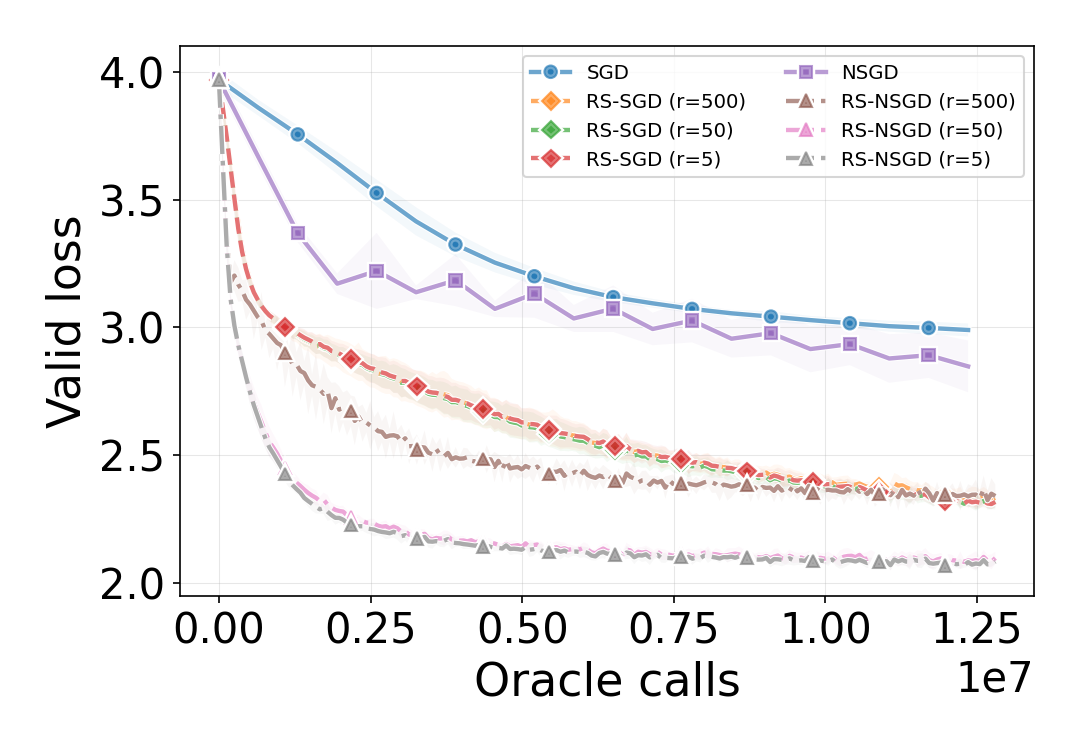}
    {\small\textbf{(a)} PTB (validation)}
  \end{minipage}\hspace{0.02\textwidth}
  \begin{minipage}[t]{0.4\textwidth}
    \centering
    \includegraphics[width=\linewidth]{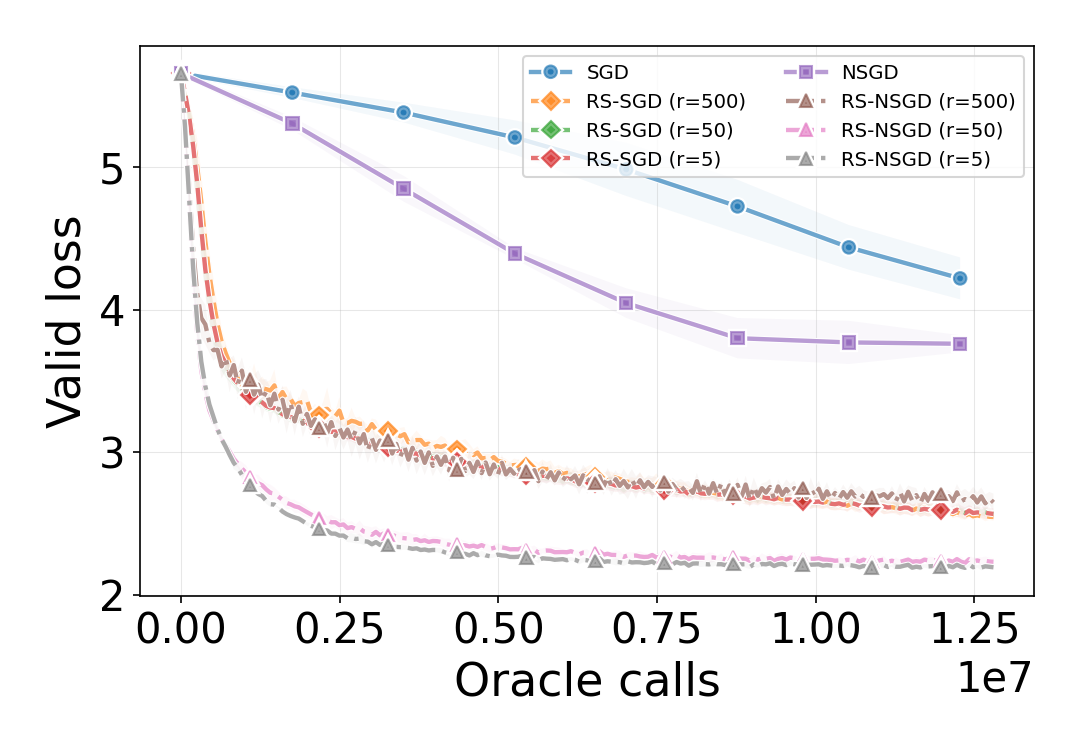}
    {\small\textbf{(b)} WikiText-2 (validation)}
  \end{minipage}

  \begin{minipage}[t]{0.4\textwidth}
    \centering
    \includegraphics[width=\linewidth]{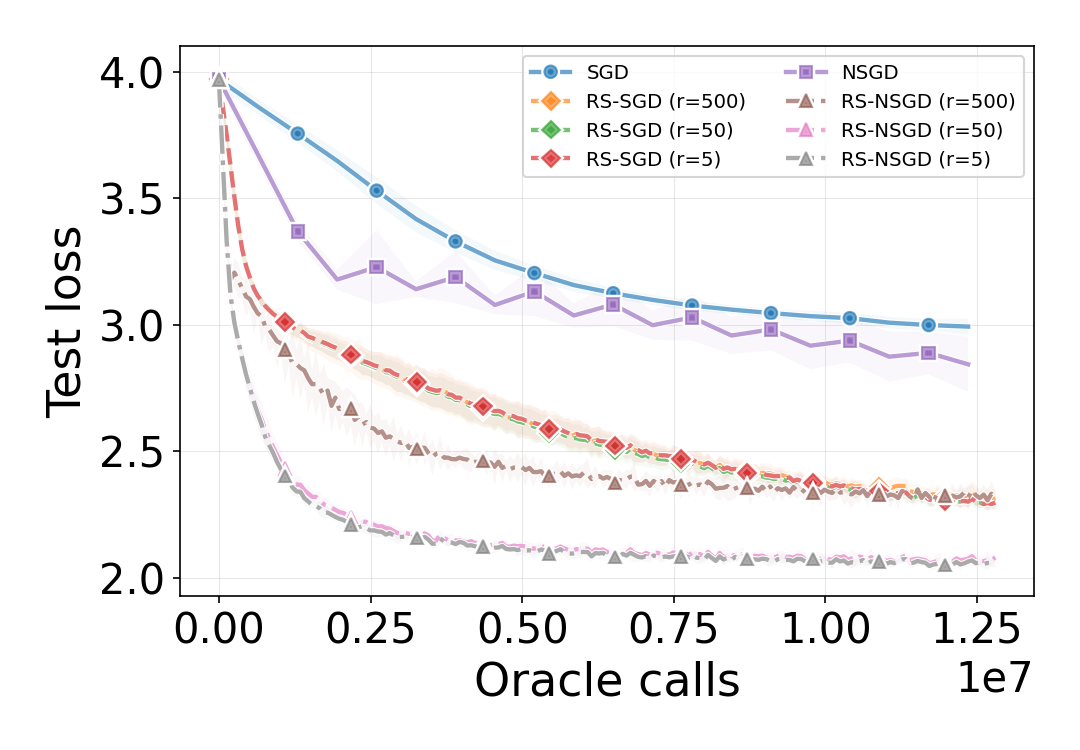}
    {\small\textbf{(c)} PTB (test)}
  \end{minipage}\hspace{0.02\textwidth}
  \begin{minipage}[t]{0.4\textwidth}
    \centering
    \includegraphics[width=\linewidth]{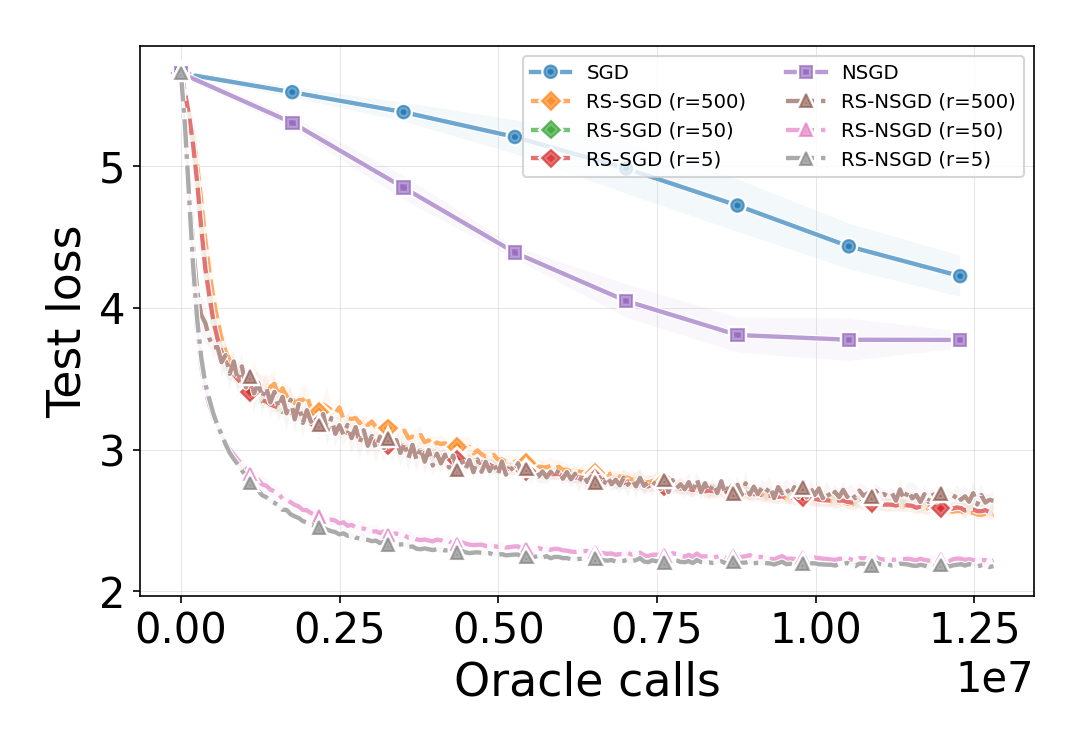}
    {\small\textbf{(d)} WikiText-2 (test)}
  \end{minipage}

  \caption{Character-level language modeling on PTB ($d=5079,\bar{B}=128$) and WikiText-2 ($d=13700,\bar{B}=128$): oracle calls vs.\ validation/test loss (5 seeds; mean $\pm$ 1 std.).}

  \label{fig:charlm-valid-test}
\end{figure*}

\end{document}